\documentclass[10pt]{article}
\usepackage[utf8]{inputenc}
\usepackage[american,british]{babel}
\usepackage{amsmath,amsthm,amsfonts,amssymb,amscd,dsfont}
\usepackage{mathrsfs}
\usepackage{mathtools}
\usepackage{stmaryrd}
\usepackage{physics}
\usepackage{subcaption}
\usepackage{graphicx}
\usepackage{hyperref}
\hypersetup{
    colorlinks=true,
    linkcolor=blue,
    citecolor=red
    }    
\usepackage[capitalize]{cleveref}  
\Crefname{equation}{}{}

\usepackage{enumitem}
\usepackage[table]{xcolor}
    
\usepackage{fullpage}
\usepackage{lastpage}  

\usepackage{tikz}
\usetikzlibrary{decorations.markings}
\usetikzlibrary{positioning}

\numberwithin{equation}{section}

\newtheoremstyle{break}%
  {}{}%
  {\itshape}{}%
  {\bfseries}{}
  {\newline}{}

\theoremstyle{break}

\newtheorem{theorem}{Theorem}[section]
\newtheorem{definition}[theorem]{Definition}

\newtheorem{lemma}[theorem]{Lemma}
\newtheorem{proposition}[theorem]{Proposition}

\crefname{maintheorem}{theorem}{theorems}

\theoremstyle{plain}
\newtheorem{remark}[theorem]{Remark}

\newcommand{\V}[1]{\mathbf{#1}}

\newcommand{\T}{\mathbb{T}}

\newcommand{\N}{\mathbb{N}}

\newcommand{\Z}{\mathbb{Z}}
\newcommand{\R}{\mathbb{R}}
\newcommand{\C}{\mathbb{C}}
\newcommand{\E}{\mathbb{E}}

\renewcommand{\P}{\mathbb{P}}

\newcommand{\Dcal}{\mathcal{D}}

\newcommand{\Gcal}{\mathcal{G}}
\newcommand{\bigO}{\mathcal{O}}

\newcommand{\Ecal}{\mathcal{E}}

\newcommand{\Xcal}{\mathcal{X}}
\newcommand{\Lcal}{\mathcal{L}}

\newcommand{\zz}{\V{z}}

\newcommand{\eps}{\epsilon}

\newcommand{\vp}{\varphi}
\renewcommand{\mod}[1]{\,(\mathrm{mod}\, #1)}

\newcommand{\df}{\operatorname{d}\!}
\newcommand{\dx}{\df x}
\newcommand{\dy}{\df y}

\renewcommand{\d}{\, \mathrm{d}}
\renewcommand{\Re}{\mathrm{Re} \,}
\renewcommand{\Im}{\mathrm{Im} \,}

\DeclareMathOperator{\dist}{\mathrm{dist}}

\DeclareMathOperator{\supp}{\mathrm{supp}}

\DeclareMathOperator{\Var}{\mathrm{Var}}

\newcommand{\rbr}[1]{\left( #1 \right)}
\newcommand{\abr}[1]{\left[ #1 \right]}

\renewcommand{\abs}[1]{\left | #1 \right |} 

\renewcommand{\norm}[1]{\left\lVert#1\right\rVert}

\newcommand{\1}{\mathbf{1}}
\newcommand{\half} {\frac{1} {2} }

\newcommand{\Unif}{\mathrm{Unif}}

\begin{document}
\title{Dirichlet $L$-functions on the critical line and multiplicative chaos}
\author{Sami Vihko$^1$}

\date{$^1$University of Helsinki, Department of Mathematics and Statistics, P.O. Box 68, FIN-00014 University of Helsinki, Helsinki, Finland \\[2ex]   
}

\maketitle
\tableofcontents

\begin{abstract}
In this article, we prove that the Dirichlet $L$-functions $L(1/2+ix,\chi_q)$, where $\chi_q$ is a uniformly random Dirichlet character modulo $q$ and $x\in \R$, converge as $q\to \infty$ to a random Schwartz distribution $\zeta_{\mathrm{rand}}$, which is related to complex Gaussian multiplicative chaos (GMC). This is the same limiting object that appeared in \cite{SaWe20a}, where the authors proved that the random shifts of the Riemann zeta function on the critical line $\zeta(1/2+ix+i\omega T)$, where $\omega\sim \mathrm{Unif} ([0,1])$, converge as $T\to \infty$. As an application of our method, we also prove that to the right of the critical line, that is, for $\Re(s)>1/2$, the convergence of the random $L$-functions $L(s,\chi_q)$ occurs in the space of analytic functions, if we exclude the zero probability (in the limit $q\to \infty$) event that $\chi_q$ is the principal character. This was already proven in \cite{Ba81a}, but our method provides a novel proof of this fact. Our results  hold for moduli $q=q_n$ such that the smallest prime divisor of $q_n$ tends to $\infty$ as $n\to\infty$.
\end{abstract}

\maketitle

\begin{section}{Introduction}
\label{sec:introduction}
This article is intended for two mathematical communities: the number theory community that employs probabilistic methods, and the probability community interested in the applications and universality of multiplicative chaos. To accommodate both perspectives, we provide a somewhat longer introduction, as we do not assume that readers are familiar with all aspects of the subject. In addition, following this philosophy, we have written out the preliminaries in \Cref{sec:some preliminary results} in full, even though some of the material may be regarded as standard, or at least not surprising, within number theory.

The central objects of investigation in this article are the Dirichlet $L$-functions $L(s,\chi)$, where $s$ is a complex variable\footnote{We shall use the common convention in analytic number theory that $s$ instead of $z$ denotes a complex variable.}, and $\chi\colon \N\to\C$ is a function called a Dirichlet character. See \Cref{sec:Basics of characters and Lfunctions} below for proper definitions. Dirichlet characters $\chi$ depend on a positive integer $q$ called their modulus. For fixed $q$, there are several Dirichlet characters, and their number is given by the Euler totient function $\vp(q)$. In this paper, we study $L(s,\chi_q)$ with $\Re(s)\geq 1/2$, where $\chi_q$ is chosen uniformly at random from the set of Dirichlet characters modulo $q$. Our main interest is in the critical line $\Re(s)=1/2$. More precisely, we study $L(s,\chi_q)$ as $q\to \infty$. On the critical line, we prove that the limit exists in a certain sense and that the limiting object is the random generalized function $\zeta_{\mathrm{rand}}$, called the randomized zeta function. It also appeared as a limiting object in \cite{SaWe20a} in the study of the statistical behaviour of the Riemann zeta function $\zeta(s)$ on the critical line and asymptotically far from the real axis, that is, as the $T\to\infty$ limit of the random shifts $\zeta(s+i\omega T)$ with $\omega\sim \Unif([0,1])$. $\zeta_{\mathrm{rand}}$ is related to a very special random generalized function called (complex) Gaussian multiplicative chaos (GMC), see \Cref{sec:basics of GMC and description of zetarand} for more information. GMC has been ubiquitous in many models of mathematical physics and probability theory over roughly the past 15 years, see for example \cite{AsJoKu11a,DuSh11a,DaKuRh16a,ArBeHa17a,ArHaKi22a}.  

Dirichlet $L$-functions are generalizations of the famous Riemann zeta function defined by 
\begin{equation}
\label{eq:def of zeta}
\zeta(s):=\sum_{n=1}^\infty n^{-s}
\end{equation}
for $\Re(s)>1$, and analytically continued to $\C \setminus \{1\}$. They were originally used to study primes in arithmetic progressions $a+nd$, where $a,d\in\N$ are relatively prime and $n$ is a positive integer. These functions can also be defined as a series in the same half-plane as $\zeta$, that is,
\begin{equation}
\label{eq:def of L-functions}
L(s,\chi):=\sum_{n=1}^\infty \frac{\chi(n)}{n^{s}}.
\end{equation}
They share many properties with the Riemann zeta function. In particular, the $L$-functions also have an analytic continuation to $\C \setminus \{1\}$.  Series such as \Cref{eq:def of zeta} and \Cref{eq:def of L-functions} are in general called Dirichlet series. 

The study of the behaviour of such objects on the critical line $\Re(s)=1/2$ is notoriously difficult. This is no surprise since the distribution of the zeros of the Riemann zeta function in the critical strip $\Re(s)\in (0,1)$ is related to many open questions about the prime numbers. The Riemann hypothesis (RH) and its generalization to Dirichlet $L$-functions (GRH) state that all non-trivial zeros of the associated functions are located on the critical line. Both the Riemann $\zeta$-function and Dirichlet $L$-functions admit an Euler product representation, from which the connection to prime numbers is evident
\begin{align}
\label{eq:Euler product for zeta}
\zeta(s)&=\prod_p (1-p^{-s})^{-1}
\\
\label{eq:Euler product for L}
L(s,\chi)&=\prod_p(1-\chi(p)p^{-s})^{-1},
\end{align}
valid in the same region $\Re(s)>1$ as the series representations. Above, the product ranges over all primes $p$. This is easy to verify by using geometric series and the uniqueness of prime factorization.

Probabilistic methods have long been used in analytic number theory to study the generic behaviour of the Riemann zeta function. Often in this context the term used is statistical behaviour. One of the first and most important such results on the critical line is the classical central limit theorem of Selberg from the 1940s \cite{Se46a}. 
An elegant new proof of Selberg's central limit theorem was provided recently in \cite{RaSo17a}. Assuming the GRH, Selberg himself generalized this to a certain class of Dirichlet series that also contains a class of Dirichlet $L$-functions called the primitive $L$-functions \cite{Se91a}. A new proof of this result for the primitive $L$-functions was also provided recently in \cite{HsWo20a} in the spirit of \cite{RaSo17a}. These results tell us roughly that at a random point with $\Im(s)\in [T,2T]$ and large $T>0$ along the critical line $\Re(s)=1/2$, the logarithm of the $\zeta$-function and certain $L$-functions behave like a mean zero Gaussian with a variance that diverges as $T\to \infty$. The divergence is of order $\log\log(T)$. Therefore, to properly formulate a central limit theorem, the logarithm of the zeta function must be multiplied by a vanishing (as $T\to\infty$) normalization constant proportional to $(\log\log(T))^{-1/2}$.   

From the perspective of this paper, there is another relevant question. Namely, does something similar hold at any point on the critical line if we take a random Dirichlet character $\chi$ modulo $q$, consider the logarithm of $L(\cdot,\chi)$ at that point and let $q$ tend to infinity? In other words, does the logarithm of $L(1/2+ix,\chi)$ behave like a Gaussian with a variance that diverges as we take $q\to \infty$ for any $x\in\R$? This is a hard open question since it is difficult to control the zeros of Dirichlet $L$-functions as $q$ and the characters modulo $q$ vary. There has been some evidence in this direction, see for example \cite{KeSn00b,KeSn03a}. In these papers, the authors conjecture such behaviour based on connections with random matrix theory and numerical computations. See also \cite{BuEvLe25a}, where the authors prove a version of this type of central limit theorem with respect to a weighted measure, assuming the GRH. All of these consider only the special case of the central value $s=1/2$ and the class of primitive $L$-functions.

The connection between random matrix theory and the study of Riemann $\zeta$- and Dirichlet $L$-functions dates back to Montgomery and Dyson in the 1970s, who conjectured that statistically the zeros of the Riemann $\zeta$-function behave like the eigenvalues of a large Hermitian random matrix. See the book review
 \cite{Co01a} by Conrey for early developments in this direction of investigation. Since the turn of the millennium, the connection with random matrices has led to many new conjectures. We do not try to list them here, but only mention the ones that are related or similar in some sense to our results. Firstly, Keating and Snaith realized that statistically the $\zeta$-function on the critical line could be modelled by the characteristic polynomial of a random unitary matrix \cite{KeSn00a}. Based on this, they made predictions about the value distribution and moments of the $\zeta$-function. As already mentioned, the same authors also made conjectures about Dirichlet $L$-functions for large modulus $q$ \cite{KeSn00b,KeSn03a}. Then, about a decade ago, the papers \cite{FyHiKe12a,FyKe14a} by Fyodorov, Hiary and Keating, and by Fyodorov and Keating, respectively, predicted that on global and mesoscopic scales on the critical line $\Re(s)=1/2$, the logarithm of the $\zeta$-function should behave like a log-correlated random field, that is, a random generalized function with a logarithmic singularity on the diagonal of its covariance kernel. Like the GMC mentioned above, these fields are ubiquitous in many models of mathematical physics and probability theory. GMC is usually constructed from such fields. It can also be used to study some properties of these fields, for example, their maxima and some fractal properties. See, for example, the article \cite{Ma15a}, the review \cite{BaKe22a} and references therein for the maximum, and the review \cite{RhVa14a} and references therein for the fractal properties.

The first rigorous results on the functional statistical behaviour of the $\zeta$-function on the critical line were achieved in \cite{SaWe20a}. Away from the critical line, earlier results exist (see the book \cite{La96a} by Laurin\u{c}ikas and references therein). The results in \cite{SaWe20a} are unconditional, that is, they do not assume RH. The result \cite[Theorem 1.1]{SaWe20a} states that, in the sense of generalized functions or Schwartz distributions,
\begin{equation}
\zeta(1/2+ix+i\omega T)\overset{d}{\to}\zeta_{\mathrm{rand}}(1/2+ix).
\end{equation}
as $T\to \infty$. Above, $\omega\sim\Unif([0,1])$. Our main result states that, as generalized functions,
\begin{equation}
L(1/2+ix,\chi_q)\overset{d}{\to}\zeta_{\mathrm{rand}}(1/2+ix)
\end{equation}
as $q\to \infty$. More precisely, we show convergence when integrated against compactly supported smooth test functions $f$. Above, $\chi_q$ is a uniformly random Dirichlet character modulo $q$. We also do not assume the GRH. The limiting object can be understood in multiple ways. Formally, we can write it as the randomized Euler product
\begin{equation}
\label{eq:formal product for zetarand}
\zeta_{\mathrm{rand}}(s)=\prod_{p}(1-\omega_pp^{-s})^{-1},
\end{equation}
where $\omega_p$ are i.i.d.\ random variables indexed by prime numbers and uniformly distributed on the complex unit circle $\T:=\{z\in\C\mid |z|=1\}=\{e^{i\theta}\mid\theta\in [0,2\pi)\}$. For $\Re(s)>1/2$, the product above is well defined pointwise; see \Cref{sec:Convergence in the space of analytic functions} for more information. \cite[Theorem 1.6]{SaWe20a} states that $\zeta_{\mathrm{rand}}$ is an honest Schwartz distribution and not a complex function or a complex measure. Lastly, we mention that $\zeta_{\mathrm{rand}}$ can be understood as the boundary value, on the critical line $\Re(s)=1/2$ and in the sense of Schwartz distributions, of an analytic function $\zeta_{\mathrm{rand}}(s)$. The Euler product in \Cref{eq:formal product for zetarand} indeed converges for $\Re(s)>1/2$ and defines an analytic function; see \Cref{sec:Convergence in the space of analytic functions} for more information. 

The main motivation for this article is that the limiting object is related to the complex GMC; see \Cref{sec:basics of GMC and description of zetarand} for more details about this connection and further discussion. As already mentioned at the beginning, these objects are encountered in diverse situations in applications of probability theory. In a sense, they are ``universal''. Our result further supports this point of view, although it is not very surprising given that many properties of the Riemann $\zeta$-function are shared  by Dirichlet $L$-functions.  

The proof of \cite[Theorem 1.1]{SaWe20a} uses a multi-step approximation similar to ours. The main difference is that they work mostly with the Euler product representation and need to resort to the theory of one-dimensional singular integrals to work in $L^2$, whereas we can work directly in $L^2$ with the Dirichlet series representation.  

\subsection*{Structure of the rest of the paper}
In the remaining subsections of the introduction, we first review the basics of Dirichlet characters and $L$-functions (\Cref{sec:Basics of characters and Lfunctions}), then discuss Gaussian multiplicative chaos and its relation to the randomized zeta function $\zeta_{\mathrm{rand}}$ and related questions (\Cref{sec:basics of GMC and description of zetarand}), and finally present the precise statement of our main result (\Cref{thm:main theorem}) and sketch the proof. We have collected all auxiliary results in \Cref{sec:some preliminary results}. Based on these auxiliary results, we present the detailed proof of  \Cref{thm:main theorem} in \Cref{sec:proof of the main theorem}. Finally, in \Cref{sec:Convergence in the space of analytic functions}, we prove that essentially the same computations yield the convergence of the random Dirichlet $L$-functions $L(s,\chi_q)$ in the space of analytic functions on the half-plane $\Re(s)>1/2$ to the randomized zeta function, if we exclude the event $\{\chi_q=\chi_0\}$, that is, the principal character. 

\begin{subsection}{Basic definitions and properties of Dirichlet characters and $L$-functions}
\label{sec:Basics of characters and Lfunctions}
For the convenience of readers without a number theory background (like the author), we collect the basic definitions and properties needed from the theory of Dirichlet characters and $L$-functions in this section. Most of the material in this section can be found in any introductory book on (analytic) number theory, for example \cite{Ap76a}.

We begin with Dirichlet characters. Dirichlet characters modulo $q$, a positive integer, are the group characters of the group of reduced residue classes, extended periodically, with the additional property that they vanish for integers not relatively prime to $q$. To be more precise, we first specify the group of reduced residue classes. A reduced residue system modulo $q$ is the set of $\vp(q)$ pairwise incongruent ($a_i\not\equiv a_j \mod q$ if $i\neq j$) integers $\{a_1,a_2,\dots,a_{\vp(q)}\}$ such that $a_i$ is relatively prime to $q$ for all $i$. Here $\vp$ is the Euler totient function, defined by
\begin{equation}
\vp(q):=|\{k\in \N \mid k\leq q, \gcd(q,k)=1\}|.
\end{equation} 
Here $\gcd$ refers to the greatest common divisor and, as usual in number theory, we shall from now on use the shorthand notation $\gcd(x,y)\equiv (x,y)$. The residue class corresponding to an integer $k$ is the set $\hat{k}:=\{x\mid x\equiv k \mod q\}$. If we define the product of two residue classes $\hat{k}$ and $\hat{l}$ by $\hat{k}\hat{l}:=\widehat{kl}$, then the set of reduced residue classes modulo $q$ (these have the restriction that $(k,q)=1$, but we continue using the same notation in the following discussion), equipped with this product, forms an Abelian group of order $\vp(q)$. The class $\hat{1}$ is the identity, and the inverse of $\hat{k}$ is the class $\hat{l}$ such that $kl\equiv 1 \mod q$.

Recall that, for Abelian groups, a character is a homomorphism from the group to the complex numbers that is not identically zero. For finite groups, as in the following, the values lie on the unit circle $\T$. Then we have:
\begin{definition}[Dirichlet characters]
Let $G$ be the group of reduced residue classes modulo $q$. For each character $f$ of $G$, we define the corresponding function $\chi\colon \N \to \mathbb{T}\cup\{0\}$.
\begin{align}
\chi(n)&:=f(\hat{n}), \quad \text{ if $(n,q)=1$},
\\
\chi(n)&:=0, \quad\,\,\, \quad  \text{ if $(n,q)>1$}.
\end{align}
The function $\chi$ is called a Dirichlet character modulo $q$. The principal Dirichlet character is 
\begin{equation}
\chi_0(n)=
\begin{cases}
1, \quad \text{if $(n,q)=1$}
\\
0, \quad \text{if $(n,q)>1$}.
\end{cases}
\end{equation}
\end{definition}

There are $\vp(q)$ distinct Dirichlet characters modulo $q$. They are completely multiplicative and $q$-periodic, that is, for any $m$ and $n$,
\begin{align}
\label{eq:multiplicativity}
\chi(mn)&=\chi(n)\chi(m)
\\
\label{eq:periodicity}
\chi(n+q)&=\chi(n),
\end{align}
see for example \cite[Theorem 6.15]{Ap76a}.
Furthermore, Dirichlet characters satisfy the following orthogonality relation. Let $\chi_i$, $i=0,1,\dots,\vp(q)-1$, be the Dirichlet characters modulo $q$, and let $m,n$ be two integers with $(n,q)=1=(m,q)$. Then we have 
\begin{equation}
\label{eq:orthogonality relation for Dirichlet characters}
\sum_{i=0}^{\vp(q)-1}\chi_i(n)\overline{\chi_i(m)}=
\begin{cases}
\vp(q), \quad \text{ if $m\equiv n \mod q$},
\\
0, \quad\quad\,\,\, \text{ if $m\not \equiv n \mod q$, }
\end{cases}
\end{equation}
see for example \cite[Theorem 6.16]{Ap76a}. Note that if $(n,q)>1$ or $(m,q)>1$, the result of the above sum is always zero. 

Next, let us consider Dirichlet $L$-functions. Proofs of all the properties listed below, except the last one, can be found in \cite[Chapter 12]{Ap76a}. Dirichlet $L$-functions depend on the integer $q$ through the Dirichlet characters modulo $q$. Given a Dirichlet character $\chi$ modulo $q$, we define for $\Re(s)>1$ 
\begin{equation}
\label{eq:series definition of L-functions}
L(s,\chi):=\sum_{n=1}^\infty\frac{\chi(n)}{n^s}. 
\end{equation}
This series converges absolutely and defines an analytic function in this half-plane for any character $\chi$, including the principal one $\chi=\chi_0$. For the principal character $\chi=\chi_0$, the region of convergence of the series \Cref{eq:series definition of L-functions} cannot be enlarged. However, for non-principal characters $\chi\neq\chi_0$, the series \Cref{eq:series definition of L-functions} converges in the half-plane $\Re(s)>0$, but the convergence is absolute only for $\Re(s)>1$. Nevertheless, the convergence is uniform on compact subsets of the full region of convergence $\Re(s)>0$ for non-principal characters $\chi\neq \chi_0$.

Furthermore, for any Dirichlet character $\chi$, the $L$-function $L(\cdot,\chi)$ can be continued to a meromorphic function in the whole complex plane $\C$. For non-principal characters ($\chi\neq\chi_0$), this continuation is even an entire function. The $L$-function $L(\cdot,\chi_0)$ corresponding to the principal character shares the analytic properties of the Riemann zeta function $\zeta$. Namely, it is meromorphic with a single pole at $s=1$, which is of first order. Furthermore, we can express the $L$-function $L(s,\chi_0)$ in terms of the zeta function
\begin{equation}
\label{eq:principal Lfunction with zeta function}
L(s,\chi_0)=\zeta(s)\prod_{p|q}(1-p^{-s}),
\end{equation}
where the product runs over all distinct prime factors of $q$. 
This last claim follows easily from the Euler product forms \Cref{eq:Euler product for zeta} and \Cref{eq:Euler product for L} in their common region of validity, and then by analytic continuation.

This is all we need from the theory of Dirichlet characters and the corresponding $L$-functions.
\end{subsection}

\begin{subsection}{Short introduction to multiplicative chaos and the randomized zeta function $\zeta_{\mathrm{rand}}$ on the critical line.}
\label{sec:basics of GMC and description of zetarand}
For the convenience of readers with a number theory background who may not be familiar with the theory of multiplicative chaos, we review only the definitions and properties of multiplicative chaos that are essential for understanding the limiting object of our main result. 

Multiplicative chaos is formally the exponential $e^X$ of a specific random field $X$ on a subset of $\R^d$. If $X$ were an honest random function, say a.s.\ continuous, there would be no difficulty in defining the exponential. However, this is not the case for the most important class of fields.  

The class of random fields most relevant for GMC is the class of log-correlated Gaussian fields. We first review the real case (with $d=1$, since this is the case relevant to us). Given a domain, say an interval $I$, a real log-correlated Gaussian field $X$ on $I$ is a generalized random field with covariance kernel $C\colon I\times I\to\R$ given by
\begin{equation}
C(x,y)=\log(|x-y|^{-1})+g(x,y)
\end{equation}
with $g$, say, continuous and bounded, such that for all suitable real test functions $f$ the evaluation $X(f)$ is a real-valued Gaussian random variable with variance 
\begin{equation}
\Var(X(f)):=\int_{I\times I}f(x)f(y)C(x,y)\dx\dy.
\end{equation}
Thus, $X$ is not an ordinary random function, but rather a random element in the space of generalized functions or Schwartz distributions. For the construction of such a field, see for example \cite[Section 2.1]{JuSaWe20a}, where the authors prove that the zero extension (meaning $X=0$ on $\R\setminus I$) of such a field exists as a random element in the Sobolev space $H^{-\eps}(\R)$ for all $\eps>0$.

Therefore, the exponential is not, a priori, well-defined. To define multiplicative chaos, we first need to regularize the field and then take a limit (an alternative and more general point of view is given in \cite{Sh16a}, where GMC is characterized via a Cameron--Martin--Girsanov property). It turns out that, in the real case, multiplicative chaos can be understood as a random measure. 

For the proofs of the following statements in the subcritical and critical cases, we refer to the papers \cite{RoVa10a,Be17a,DuRhSh14b} and the review papers \cite{RhVa14a,Po20a} and references therein. Given a suitable sequence $\{X_N\}_{N\in\N}$ of smooth approximations of the field $X$, the sequence of random measures $\mu_{\gamma, N}(\dx):=\frac{e^{\gamma X_N}}{\E[e^{\gamma X_N}]}\dx$ with $\gamma\in\R$ converges to a random measure in the weak topology of the space of Radon measures. The stochastic nature of the convergence may depend on the approximating sequence, but the limiting object does not in most cases. The limiting measure is unique, but non-trivial only for $|\gamma|<\sqrt{2}$, and non-atomic in the same range. In the limiting case $|\gamma|=\sqrt{2}$, further renormalization is required to obtain a non-trivial limit. For $|\gamma|>\sqrt{2}$, yet another procedure is needed, and the limiting measure is then purely atomic. 

The rigorous study of multiplicative chaos was initiated by Kahane in 1985 \cite{Ka85a}. The original motivation lay in the rigorous study of Kolmogorov's theory of turbulence in fluid dynamics and in constructing a continuous analogue of Mandelbrot's multiplicative cascades \cite{Ma74a,Ma74b}. 
We refer the reader to the nice review \cite{RhVa14a} and the elegant construction \cite{Be17a} in the sub-critical ($|\gamma|<\sqrt{2}$) case and to the review \cite{Po20a} and references therein for the critical $(|\gamma|=\sqrt{2})$ case. Note that the distribution of $X$ is symmetric, so it suffices to consider $\gamma>0$.

To describe the limiting object $\zeta_{\mathrm{rand}}$ in our main theorem in terms of multiplicative chaos, we need to consider the complex case. In this case, the limiting object is no longer necessarily a measure, but merely a Schwartz distribution. In the complex case, we can either consider complex parameters $\gamma$ or allow the field $X$ to be complex. The latter case is studied in \cite{LaRhVa15a}. Actually, they study the case where the real and imaginary parts of the field have separate real parameters, that is $\gamma X$ is replaced by $\gamma_1X_1+i\gamma_2X_2$ with $\gamma_i\in\R$ and $X_i$ real log-correlated fields. In our case, we would have $\gamma_1=1=\gamma_2$. In \cite{LaRhVa15a} the authors assume that the real ($X_1$) and imaginary ($X_2$) parts of the field are independent. However, this assumption is inadequate for describing $\zeta_{\mathrm{rand}}$. 

Instead, we need a very specific covariance structure for the real and imaginary parts. By \cite[Theorem 1.1 (ii)]{SaWe20a}, we can write, as Schwartz distributions,
\begin{equation}
\label{eq:zetarand in terms of GMC}
\zeta_{\mathrm{rand}}(1/2+ix)=g(x)\nu(x),
\end{equation}
where $g$ is a smooth random function on $\R$ with almost surely no zeros, and ``$\nu(x)=e^{\mathcal{G}(x)}$'' is the multiplicative chaos corresponding to the Gaussian field $\mathcal{G}$ with the covariance structure
\begin{align}
\E[\mathcal{G}(x)\mathcal{G}(y)]&=0
\\
\E[\mathcal{G}(x)\overline{\mathcal{G}}(y)]&=\log(\zeta(1+i(x-y)))
\end{align}
or equivalently
\begin{align}
\E\abr{\Re(\mathcal{G})(x)\Re(\mathcal{G})(y)}&=\E\abr{\Im(\mathcal{G})(x)\Im(\mathcal{G})(y)}=\half\log\abs{\zeta(1+i(x-y))}
\\
\E\abr{\Re(\mathcal{G})(x)\Im(\mathcal{G})(y)}&=-\E\abr{\Im(\mathcal{G})(x)\Re(\mathcal{G})(y)}=-\half\mathrm{Arg}(\zeta(1+i(x-y)))
\end{align}
for $x,y\in\R$. This translates to
\begin{equation}
\E[\Re(\mathcal{G})(x)\Im(\mathcal{G})(y)]=\frac{\pi}{4}\mathrm{sgn}(x-y)+\text{smooth terms}. 
\end{equation}
A similar type of complex multiplicative chaos has also appeared, for example, in \cite{NaPaSi23a} in the context of random unitary matrices. 

To prove \Cref{eq:zetarand in terms of GMC} the authors of \cite{SaWe20a} use the truncated randomized Euler product 
\begin{equation}
\zeta_{N,\mathrm{rand}}(s):=\prod_{p\leq N}(1-\omega_p p^{-s})^{-1},
\end{equation}
where $\omega_p\sim \Unif(\T)$ are i.i.d.\ random variables, and the following decomposition:
\begin{equation}
\log(\zeta_{N,\mathrm{rand}}(1/2+ix))=\Gcal_N(x)+\Ecal_N(x),
\end{equation}
where $\Gcal_N$ is an explicit Gaussian process on $[-A,A]$ for any $A>0$, and $\Ecal_N$ is a smooth random function that almost surely converges uniformly to a smooth random function $\Ecal$ \cite[Theorem 1.7]{SaWe20a}. This decomposition allows them to show that $\zeta_{N,\mathrm{rand}}$ also gives rise to real multiplicative chaos. Namely, they show that the measures $C(N)|\zeta_{N,\mathrm{rand}}(\half+ix)|^\gamma/\E[|\zeta_{N,\mathrm{rand}}(\half+ix)|^\gamma] \dx$, where $C(N)$ is a renormalization constant (equal to $1$ unless $\gamma=2$\footnote{Note that this is $\sqrt{2}\gamma_c$, $\gamma_c$ is the critical value given above in the context of standard GMC. This is because here the leading behaviour of the covariance is $\half\log(|x-y|)^{-1}$.}), converge to a non-trivial limiting measure for all $0<\gamma\leq 2$. This measure can be described as the product of a well-behaved random function and a real GMC \cite[Theorems 1.8 and 1.9]{SaWe20a}.

It is still an open question whether $\log(\zeta(1/2+ix+i\omega T))$ also converges to $\Gcal(x)+\Ecal(x)$. Similarly, it is still only a conjecture that $\mu_T(x):=\zeta(1/2+ix+i\omega T)$ gives rise to real multiplicative chaos, that is, whether the measures $C(T)|\mu_T(x)|^\gamma/\E[|\mu_T(x)|^\gamma]\dx$, where $C(T)$ is again a renormalizing constant only needed in the critical case, converge to a non-trivial limiting measure. The latter is likely a harder question since we would need to control the moments $\E[|\mu_T(x)|^\gamma]$ around a general point $x$. Similar questions for the functions $L(1/2+ix,\chi_q)$ are also out of reach, mainly because we lack the means to control the proportion of the characters for which $L(1/2+ix,\chi)\neq 0$ for a given $x\in\R$, even for a fixed modulus $q$, let alone in the limit $q\to \infty$. There are some results, but they often consider only primitive $L$-functions at the central point $s=1/2$; see, for example, the recent preprint \cite{Ea26a} and references therein. The moments $\E[|L(1/2+ix,\chi_q)|^\gamma]$ are also not sufficiently well understood for $x\neq 0$ and the relevant values of $\gamma$.

The moments of both the zeta function in the $T$-aspect and various families of $L$-functions in the $q$-aspect have been widely studied, but mainly at the central value $x=0$. The classical study of the asymptotics of the moments of the zeta function in the $T$-aspect goes back to Hardy and Littlewood (second moment) and Ingham (fourth moment). For $L$-functions, the exact asymptotics of the fourth moment have been considered more recently, for example, in \cite{Yo11a,Za19a,Wu22a}. Upper and lower bounds for the moments of $L$-functions have also been studied extensively; see, for example, \cite{RuSo05a,RuSo06a,So09a,RaSo12a,Ha13b,HeRaSo19a,ArCr25a,GaZh25a,Ga24a} and references therein. Most of the $L$-function results hold only for subfamilies such as the primitive or quadratic $L$-functions, and some results are conditional, that is, they assume the RH or GRH. 

The references above on the moments and non-vanishing of $L$-functions on the critical line only scratch the surface of the literature; we merely aim to give a flavor of the limitations of current results.  
\end{subsection}

\begin{subsection}{The main result}
\label{sec:The main result}
Our main theorem is:
\begin{theorem}
\label{thm:main theorem}
Let $q$ be a positive integer, and denote by $\Xcal_q$ the set of all Dirichlet characters modulo $q$. Assume that $\chi_q$ is a uniformly random Dirichlet character, that is, $\chi_q\sim\mathrm{Unif}(\Xcal_q)$ and $f\in C_c^\infty(\R)$. Then
\begin{equation}
L_q(f):=\int_{\R}f(x)L(1/2+ix,\chi_q)\dx\overset{d}{\to}\zeta_{\mathrm{rand}}(f)
\end{equation}
as $q\to \infty$ along the prime numbers or more generally along any sequence $\{q_n\}_{n\in\N}$ in the class $Q$ defined in \Cref{rem:q sequences}. That is, we take $q_n$ to be such that the smallest prime factor in its prime decomposition tends to infinity as $n\to \infty$.
\end{theorem}

In the proof, we will use the fact, established in \cite{SaWe20a}, that the truncated randomized Euler product integrated against a test function $f$,
\begin{equation}
\label{eq:def of randomized truncated Euler product}
\zeta_{M,\mathrm{rand}}(f):=\int_\R f(x)\prod_{p\leq M}(1-p^{-(\half+ix)}\omega_p)^{-1}\dx,
\end{equation}
converges almost surely to $\zeta_{\mathrm{rand}}(f)$. In fact, we only require convergence in law. Above, $M\in\N$, and the random variables $\omega_p$, indexed by the primes $p$, are i.i.d.\ and uniformly distributed on the unit circle $\T\subset \C$. 
We will show that, for $M=M(f)\in\N$ large enough,
\begin{equation}
\label{eq:def of LcalM}
\Lcal_{M,q}(f):= \int_\R f(x)\sum_{n=1}^M\frac{\chi_q(n)}{n^{\half+ix}}\dx
\end{equation}
approximates in law the quantity 
\begin{equation}
\label{eq:def of lM}
\Lcal_{M,\omega}(f):=\int_\R f(x)\sum_{n=1}^M\frac{\omega_n}{n^{\half+ix}}\dx
\end{equation}
and that this, in turn, approximates $\zeta_{M,\mathrm{rand}}(f)$ well. The random variables $\omega_n$ for non-prime integers will be introduced in the next section. Finally, we use the approximation result given in \Cref{thm:approximation}.

\begin{remark}
Since the convergence is only in law, we cannot directly infer convergence in the topology of the space of ordinary distributions $\Dcal'(\R)$, that is, the topological dual of $\Dcal(\R)=C_c^\infty(\R)$. It is natural to expect convergence in some distribution space, e.g.\ the Sobolev space $H^{-\alpha}(\R)$ for some $\alpha>0$. However, since we already know by \cite[Theorem 1.1]{SaWe20a} that the limiting object $\zeta_{\mathrm{rand}}$ is in a weighted Sobolev space $H^{-\alpha}$ for all $\alpha>1/2$, we do not pursue this further.
\end{remark}
After we have proven the main result, as an application of our methods, we will prove an analogous result in the space of analytic functions $H(D)$ on the half-plane $D=\{s\in\C\mid \Re(s)>1/2\}$. In this case, we must either exclude the event $\{\chi_q=\chi_0\}$, that is, the principal character, or exclude the point $s=1$, where $L(s,\chi_0)$ has a first order pole. We choose to exclude the principal character since the probability $\P(\chi_q=\chi_0)=1/\vp(q)$ tends to zero as $q\to\infty$, so this event becomes negligible in the limit. Furthermore, the limiting object is in any case analytic in all of $D$. This is a version of Bagchi's theorem, which he established in his PhD thesis in 1981 \cite{Ba81a}; see also, for example, \cite[Section 3.2]{Ko21a}. The best-known version of Bagchi's theorem states that the random shifts of the Riemann zeta function $\zeta(s+it)$, where $t$ is sampled uniformly from the interval $[-T,T]$, converge in law to the randomized zeta function $\zeta_{\mathrm{rand}}$ as $T\to \infty$, as analytic functions on $D_1:=\{1/2<\Re(s)<1\}$ and $D_2:=\{\Re(s)>1\}$.

\end{subsection}
\end{section}

\begin{section}{Some preliminary lemmas}
\label{sec:some preliminary results}
Let us now consider the key property of $\chi_q$ relevant to our analysis. First, we have
\begin{lemma}
\label{lem:Convergence of the random characters evaluated at distinct primes}
For arbitrary $N\in\N$, let $\{p_i\}_{i=1,2,\dots,N}$ be a set of distinct prime numbers. Then
\begin{equation}
(\chi_q(p_1),\chi_q(p_2),\dots,\chi_q(p_N))\overset{d}{\to} (\omega_{p_1},\omega_{p_2},\dots,\omega_{p_N})
\end{equation}
as $q\to \infty$ along the prime numbers or, more generally, along any sequence in the class $Q$ defined in the remark below. The variables $\omega_{p_i}$ are i.i.d.\ with $\omega_{p_i}\sim \Unif(\T)$, that is, $\omega_{p_i}$ is uniformly distributed on the unit circle in the complex plane.
\end{lemma}

\begin{remark}
\label{rem:q sequences}
We thank the anonymous referees for pointing out that we should be more careful about the way the limit $q\to \infty$ is taken in this article. In particular, we thank them for pointing out the simple fact that \Cref{lem:Convergence of the random characters evaluated at distinct primes} cannot hold along all sequences $q_n\overset{n\to\infty}{\longrightarrow}\infty$, since along some sequences the components converge to zero for suitably chosen primes. 

Taking $q\to\infty$ along the prime numbers presents no such difficulty. However, since the rest of our argument holds for all sequences $q_n$, we also specify a class of sequences for which \Cref{lem:Convergence of the random characters evaluated at distinct primes} also holds. Consider the prime decomposition $q_n=\prod_{i}(p_i^{(n)})^{\alpha_{i}^{(n)}}$ and denote $\underline{p}_n:=\min_i\{p_i^{(n)}\}$. Then let 
\begin{equation}
Q:=\{\{q_n\}_{n\in \N}\mid \underline{p}_n\to\infty \text{ as } n\to \infty\}.
\end{equation}
\end{remark}
\begin{proof}
Since the random variables involved are bounded, it suffices to consider their moments. Knowing the moments allows us to compute the integrals of all polynomials (in the variables and their complex conjugates) and, by a density argument, also those of all continuous functions. Therefore, the measure is uniquely determined. Let us first consider, for fixed $k_i,m_i\in \Z_{\geq 0}$, $i=1,2,\dots,N$,
\begin{equation}
\begin{split}
\label{eq:moments of omegap}
\E[\omega_{p_1}^{k_1}\overline{\omega_{p_1}}^{m_1}\omega_{p_2}^{k_2}\overline{\omega_{p_2}}^{m_2}\dots\omega_{p_N}^{k_N}\overline{\omega_{p_N}  }^{m_N}]&=\frac{1}{(2\pi)^N}\prod_{j=1}^N\int_0^{2\pi}e^{i(k_j-m_j)\theta_j}\d\theta_j
\\
&=
\begin{cases}
1, &\text{if $k_i=m_i$ for all $i=1,2,\dots,N$}
\\
0, &\text{otherwise}.
\end{cases}
\end{split}
\end{equation}

On the other hand, for primes $q$, we have
\begin{equation}
\label{eq:moments of chi_q evaluated at primes}
\begin{split}
\E&\abr{[\chi_q(p_1)]^{k_1}\overline{\chi_q(p_1)}^{m_1}[\chi_q(p_2)]^{k_2}\overline{\chi_q(p_2)}^{m_2}\dots[\chi_q(p_N)]^{k_N}\overline{\chi_q(p_N)}^{m_N}}
\\
&=\frac{1}{\vp(q)}\sum_{j=0}^{\vp(q)-1}[\chi^{(j)}(p_1)]^{k_1}\overline{\chi^{(j)}(p_1)}^{m_1}[\chi^{(j)}(p_2)]^{k_2}\overline{\chi^{(j)}(p_2)}^{m_2}\dots[\chi^{(j)}(p_N)]^{k_N}\overline{\chi^{(j)}(p_N)}^{m_N}
\\
&=\frac{1}{\vp(q)}\sum_{j=0}^{\vp(q)-1}\chi^{(j)}(p_1^{k_1}p_2^{k_2}\dots p_N^{k_N})\overline{\chi^{(j)}(p_1^{m_1}p_2^{m_2}\dots p_N^{m_N})}
\\
&=
\begin{cases}
1, \text{ if $p_1^{k_1}p_2^{k_2}\dots p_{N}^{k_N}\equiv p_1^{m_1}p_2^{m_2}\dots p_{N}^{m_N}$ (mod $q$)}
\\
0, \text{ otherwise}.
\end{cases}
\end{split}
\end{equation}
Here $\chi^{(j)}$, $j=0,1,\dots,\vp(q)-1$, are the Dirichlet characters modulo $q$. We have also used the orthogonality relation of Dirichlet characters \Cref{eq:orthogonality relation for Dirichlet characters} and their multiplicative property \Cref{eq:multiplicativity}. Note that above we index with superscripts rather than subscripts to distinguish these characters from $\chi_q$. Since the exponents are fixed, the RHS of \Cref{eq:moments of chi_q evaluated at primes} converges to $\1_{\{k_i=m_i,\text{ for all } i\}}$ as $q\to \infty$ by the uniqueness of prime factorization. This establishes the convergence of the moments along the primes.

Furthermore, for sequences $\{q_n\}_{n\in\N}$ in the class $Q$, we can consider $n$ so large that $\underline{p}_n> p_1,p_2,\dots,p_N$, so that $\mathrm{gcd}(q_n,p_i)=1$ for all $i=1,2,\dots,N$. This is possible since we eventually want to take $n\to \infty$. Thus, the same argument above also holds in this case.  
\end{proof}

Since $\chi(1)=1$ for any Dirichlet character $\chi$, it is natural to also define $\omega_1:=1$. For composite integers $n=\prod_{i=1}^kp_i^{\alpha_i}$, we then define the random variables
\begin{equation}
\label{eq:def of omegan for nonprime n}
\omega_n:=\prod_{i=1}^k\omega_{p_i}^{\alpha_i}.
\end{equation}
Then we have:
\begin{lemma} 
\label{lem:convergence of characters for general integers}
For $N,n_1,n_2,\dots,n_N\in\N$ we have
\begin{equation}
(\chi_q(n_1),\chi_q(n_2),\dots,\chi_q(n_N))\overset{d}{\to}(\omega_{n_1},\omega_{n_2},\dots,\omega_{n_N}),
\end{equation}
as $q\to \infty$ along the prime numbers or along sequences in the class $Q$ defined in \Cref{rem:q sequences} above.

Furthermore, the collection $\{\omega_i\}_{i\in\N}$ satisfies the following moment structure. For $k_i,m_i\in\N$ and $i=1,2,\dots,N$ we have 
\begin{equation}
\E[\omega_{n_1}^{k_1}\overline{\omega_{n_1}}^{m_1}\omega_{n_2}^{k_2}\overline{\omega_{n_2}}^{m_2}\dots\omega_{n_N}^{k_N}\overline{\omega_{n_N}}^{m_N}]=
\begin{cases}
1, &\text{if the exponents of distinct primes in the}
\\
&\text{prime factorization of $n_1^{k_1}n_2^{k_2}\dots n_N^{k_N}$ and}
\\ 
&\text{$n_1^{m_1}n_2^{m_2}\dots n_N^{m_N}$ match.}
\\
0, &\text{otherwise.}
\end{cases}
\end{equation}
In particular, $\omega_{n}$ and $\omega_{m}$ are orthogonal, that is,
\begin{equation}
\E[\omega_n\overline{\omega_m}]=\delta_{n,m}.
\end{equation}
\end{lemma}

\begin{proof}
For primes $q$, convergence follows from the multiplicative property of Dirichlet characters \Cref{eq:multiplicativity} and the continuous mapping theorem; see, for example, \cite[Theorem 4.27]{Ka02a}. The moment structure and orthogonality are evident from the computation in \Cref{eq:moments of omegap}.

For sequences $\{q_n\}\in Q$, we must first consider $n$ so large that $\underline{p}_n$ is greater than all prime factors in the decomposition of $n_i$ for all $i=1,2,\dots,N$. 
\end{proof}

Next, we present the Fourier representation of $L_q(f)$ defined in \Cref{thm:main theorem}. 
\begin{lemma}
\label{lem:Fourier representation of Lf}
For $\chi_q\neq\chi_0$ and $f\in C_c^\infty(\R)$, we have
\begin{equation}
L_q(f)=\sum_{n=1}^\infty \frac{\chi_q(n)}{n^\half}\hat{f}(\log(n))=\lim_{N\to\infty}\sum_{n=1}^N\frac{\chi_q(n)}{n^\half}\hat{f}(\log(n)),
\end{equation}
where
\begin{equation}
\hat{f}(k)=\int_{\R}f(x)e^{-ixk}\dx
\end{equation}
is the Fourier transform of $f$.
\end{lemma}

\begin{proof}
Recall that the series defining the $L$-function $L(s,\chi)$, for non-principal characters $\chi$, converges uniformly on compact subsets of the half-plane $\Re(s)>0$. Thus, since the support of $f$ is compact, we simply compute
\begin{equation}
\begin{split}
L_q(f)&=\int_{\R}f(x)\sum_{n=1}^\infty\frac{\chi_q(n)}{n^{\half+ix}}\dx
=\sum_{n=1}^\infty\frac{\chi_q(n)}{n^{\half}}\int_{\R}f(x)e^{-ix\log(n)}\dx=\sum_{n=1}^\infty\frac{\chi_q(n)}{n^\half}\hat{f}(\log(n)).
\end{split}
\end{equation}
This completes the proof.
\end{proof}

By the smoothness of $f$, we have for $n\in\N$
\begin{equation}
\label{eq:log bounds on the Fourier transform of $f$}
\abs{\hat{f}(\log(n))}\leq
\begin{cases}
C_k(\log(n))^{-k}, \,\,\,\,\text{ if $n\geq 2$}
\\
|\hat{f}(0)|, \qquad\qquad \text{ if $n=1$}
\end{cases}
\end{equation}
for all $k\in \N$ with some constants $C_k>0$. Furthermore, for later use, we note that we can always also use the cruder estimate
\begin{equation}
\label{eq:constant bound on the Fourier transform of $f$}
\abs{\hat{f}(\log(n))}\leq C
\end{equation}
for some constant $C>0$ and all $n\in\N$.
In addition, recall that the series $\sum_{n=2}^\infty1/(n(\log(n))^a)$ converges for any $a>1$. Together, these two facts yield the following estimate, which we will need later:
\begin{equation}
\label{eq:convergence of a series of Fourier transforms of f}
\sum_{n=1}^\infty\frac{|\hat{f}(\log(n))|^2}{n}<\infty.
\end{equation}

We also require the following summation estimates.  
\begin{lemma}
\label{lem:estimates for r and t sums 1}
For any function $f\colon \N\to \C$ and $m,s\in\Z_{\geq 0}$, we have
\begin{equation}
\sum_{r,t=1}^{q-1}\1_{(mq+r,q)=1}\1_{(sq+t,q)=1}\rbr{\delta_{r,t}-\frac{1}{\vp(q)}}f(r)=0,
\end{equation}
where $\vp$ is the Euler totient function. Note that this includes the case that $f(r)$ is constant. 

Furthermore, for any integers $a,b\geq 0$ and $\sigma\geq 1/2$, we have
\begin{equation}
\frac{1}{q^{2\sigma}}\sum_{r,t=1}^{q-1}\1_{(mq+r,q)=1}\1_{(sq+t,q)=1}\abs{\rbr{\delta_{r,t}-\frac{1}{\vp(q)}}\rbr{\frac{r}{q}}^a\rbr{\frac{t}{q}}^b}=\bigO(q^{-2(\sigma-\half)}),
\end{equation}
where the implied constant is independent of $q$.
Lastly, for any integers $a,b\geq 0$ and $\sigma\geq 1/2$, we have
\begin{equation}
\frac{1}{\vp(q)q^{\sigma}}\sum_{r,t=1}^{q-1}\1_{(mq+r,q)=1}\1_{(sq+t,q)=1}\frac{1}{r^{\sigma}}\rbr{\frac{r}{q}}^a\rbr{\frac{t}{q}}^b=\bigO(q^{-(\sigma-\half)}),
\end{equation}
where the implied constant is independent of $q$.
\end{lemma}

\begin{remark}
In both the first and last statements, the roles of $r$ and $t$ may evidently be interchanged. Furthermore, note that $\sigma=\Re(s)$ as in \Cref{sec:Convergence in the space of analytic functions} and the claim is formulated to apply directly both on the critical line and in the half-plane to its right.  
\end{remark}

\begin{proof}
Since we clearly have $(q,r)=1$ if and only if $(mq+r,q)=1$, for any $m\in\N$ and $r\leq q-1$,  we will henceforth replace $\1_{(mq+r,q)=1}$ with $\1_{(r,q)=1}$ throughout. The first statement is a simple computation
\begin{equation}
\begin{split}
\sum_{r,t=1}^{q-1}\1_{(r,q)=1}\1_{(t,q)=1}\rbr{\delta_{r,t}-\frac{1}{\vp(q)}}f(r)&=\sum_{r=1}^{q-1}\1_{(r,q)=1}f(r)
-\underbrace{\rbr{\frac{1}{\vp(q)}\sum_{t=1}^{q-1}\1_{(t,q)=1}}}_{=1}\sum_{r=1}^{q-1}\1_{(r,q)=1}f(r)
\\
&=0.
\end{split}
\end{equation}

To prove the second statement, we bound each term separately. Let us write
\begin{equation}
\begin{split}
\frac{1}{q^{2\sigma}}\sum_{r,t=1}^{q-1}&\1_{(r,q)=1}\1_{(t,q)=1}\abs{\rbr{\delta_{r,t}-\frac{1}{\vp(q)}}\rbr{\frac{r}{q}}^a\rbr{\frac{t}{q}}^b}
\\
&\leq\frac{1}{q^{a+b+2\sigma}}\sum_{r=1}^{q-1}\1_{(r,q)=1}r^{a+b}+\frac{1}{q^{a+b+2\sigma}\vp(q)}\sum_{r=1}^{q-1}\1_{(r,q)=1}r^a\sum_{t=1}^{q-1}\1_{(t,q)=1}t^b
\end{split}
\end{equation}
Each sum above contains $\vp(q)$ terms and each term $x^c$, where $x=r,t$, and $c=a,b,a+b$, is bounded by $q^c$. Thus, we obtain an upper bound $\bigO(\vp(q)/q^{2\sigma})=\bigO(q^{-2(\sigma-\half)})$ for both terms. Note that $\vp(q)/q\leq 1-1/q\leq 1$ for all $q\geq 2$, with equality on the left only when $q$ is a prime.

For the last claim we have
\begin{equation}
\begin{split}
\frac{1}{\vp(q)q^\sigma}&\sum_{r,t=1}^{q-1}\1_{(r,q)=1}\1_{(t,q)=1}\frac{1}{r^\sigma} \rbr{\frac{r}{q}}^a\rbr{\frac{t}{q}}^b
=\rbr{\frac{1}{q^b\vp(q)}\sum_{t=1}^{q-1}\1_{(t,q)=1} t^b}\rbr{\frac{1}{q^{\sigma+a}}\sum_{r=1}^{q-1}\1_{(r,q)=1}r^{a-\sigma}}.
\end{split}
\end{equation}

The first factor is $\bigO(1)$ uniformly in $q$ since the sum again has $\vp(q)$ terms, which are bounded by $q^b$. Next, consider the remaining sum. Here we may replace the indicator with $1$ to obtain an upper bound
\begin{equation}
\begin{split}
\label{eq:sum for 1/sqrt r}
\frac{1}{q^{\sigma+a}}\sum_{r=1}^{q-1}r^{a-\sigma}\leq\frac{1}{q^\sigma}\sum_{r=1}^{q-1}\frac{1}{\sqrt{r}}\leq\frac{1}{q^\sigma}\int_0^{q-1}\frac{\dx}{\sqrt{x}}=q^{-(\sigma-\half)}\frac{2\sqrt{q-1}}{\sqrt{q}}
=\bigO(q^{-(\sigma-\half)}),
\end{split}
\end{equation} 
where the first inequality is based on the observations $r^a\leq q^a$ and $r^{-(\sigma-\half)}\leq 1$ for all $r=1,2,\dots,q-1$.
\end{proof}
\end{section}

\begin{section}{Proof of \Cref{thm:main theorem}}
\label{sec:proof of the main theorem}
The proof is divided into three parts, which are expressed as the three propositions below. Given these propositions, and the fact that we have 
\begin{equation}
\label{eq:convergence in distribution of the randomized zeta function}
\zeta_{N,\mathrm{rand}}(f)\overset{d}{\to}\zeta_{\mathrm{rand}}(f)
\end{equation}
as mentioned in \Cref{sec:The main result}, the result follows by applying \Cref{thm:approximation} below twice with $(X,d)=(\C,|\cdot|)$. 

\begin{theorem}[{{\cite[Theorem 4.28]{Ka02a}}}]
\label{thm:approximation}
Let $\xi$, $\xi_q$, $\eta^M$ and $\eta_q^M$ with $q,M\in\N$, be random elements in a metric space $(X,d)$ such that $\eta_q^M\overset{d}{\to}\eta^M$ as $q\to \infty$ and $\eta^M\overset{d}{\to} \xi$ as $M\to\infty$. If, in addition, the following 
\begin{equation}
\label{eq:extra condition}
\lim_{M\to \infty}\limsup_{q\to\infty}\E\abr{d(\eta_q^M,\xi_q)\wedge 1}=0
\end{equation}
holds, then $\xi_q\overset{d}{\to} \xi$ as $q\to\infty$.
\end{theorem}

We now state the three propositions to be proved.
\begin{proposition}
\label{prop:first claim of the proof of the main theorem}
Let $\chi_q\sim\Unif(\Xcal_q)$, where $\Xcal_q$ is the set of all Dirichlet characters modulo $q$ and let $f\in C_c^\infty(\R)$. Then we have 
\begin{equation}
\lim_{M\to \infty}\limsup_{q\to\infty}\E\abr{\abs{L_q(f)-\mathcal{L}_{M,q}(f)}^2}=0,
\end{equation}
where $\Lcal_{M,q}(f):=\int_{\R}f(x)\sum_{n=1}^M\frac{\chi_q(n)}{n^{\half+ix}}\dx$ and the $q$-limit is taken in the same manner as in \Cref{thm:main theorem}.
\end{proposition}

\begin{proposition}
\label{prop:second claim of the proof of the main theorem}
Let $f$ be as in \Cref{prop:first claim of the proof of the main theorem} and fix $M>0$. Then as $q\to \infty$ (in the same manner as described in \Cref{thm:main theorem}) we have
\begin{equation}
\Lcal_{M,q}(f)\overset{d}{\to}\Lcal_{M,\omega}(f):=\int_{\R}f(x)\sum_{n=1}^M\frac{\omega_n}{n^{\half+ix}}\dx,
\end{equation}
where $\Lcal_{M,q}(f)$ is as in \Cref{prop:first claim of the proof of the main theorem}, and the random variables $\omega_n$ are defined as in \Cref{eq:def of omegan for nonprime n}.
\end{proposition}

\begin{proposition}
\label{prop:third claim of the proof of the main theorem}
Let $f$ be as in \Cref{prop:first claim of the proof of the main theorem}. Then we have
\begin{equation}
\lim_{M\to\infty}\E\abr{\abs{\Lcal_{M,\omega}(f)-\zeta_{M,\mathrm{rand}}(f)}^2}=0,
\end{equation}
where $\Lcal_{M,\omega}$ is as in \Cref{prop:second claim of the proof of the main theorem} and $\zeta_{M,\mathrm{rand}}$ is the truncated randomized Euler product with
\begin{equation}
\zeta_{M,\mathrm{rand}}(f):=\int_{\R}f(x)\prod_{p\leq M}(1-\omega_p p^{-\half-ix})^{-1}\dx.
\end{equation}
\end{proposition}

\begin{remark}
The limits established in \Cref{prop:first claim of the proof of the main theorem} and \Cref{prop:third claim of the proof of the main theorem} imply the condition \Cref{eq:extra condition} in \Cref{thm:approximation} in their respective contexts.
\end{remark}

\begin{subsection}{Proof of {{\Cref{prop:first claim of the proof of the main theorem}}}}
\label{sec:proof of the first claim}
First, we have
\begin{equation}
\begin{split}
E_{q,M}&:=\E\abr{\abs{L_q(f)-\Lcal_{M,q}(f)}^2}
\\
&\,\,=\E\abr{\1_{\chi_q\neq \chi_0}\abs{\int_\R f(x)\sum_{n=M+1}^\infty \frac{\chi_q(n)}{n^{\half+ix}}\dx}^2}+\P(\chi_q=\chi_0)\abs{L_q^{(0)}(f)-\Lcal_{M,q}^{(0)}(f)}^2
\\
&\,\,=:E_{q,M}^{(1)}+E_{q,M}^{(2)},
\end{split}
\end{equation}
where $L_q^{(0)}$ and $\Lcal_{M,q}^{(0)}$ are defined as $L_q$ and $\Lcal_{M,q}$ with $\chi_q=\chi_0$, that is, the principal character. 

Let us first consider the second term $E_{q,M}^{(2)}$. Observe that $\P(\chi_q=\chi_0)=1/\vp(q)$ and $|\Lcal_{M,q}^{(0)}(f)|=\bigO_M(1)$ for fixed $M$. Thus, it is enough to show that $1/\vp(q)\to 0$ as $q\to \infty$ and
\begin{equation}
\frac{1}{\vp(q)}|L_q^{(0)}(f)|^k\overset{q\to \infty}{\longrightarrow}0
\end{equation}
for $k=1,2$. It is sufficient to prove the case $k=2$. Using the representation of the principal $L$-function via the Riemann $\zeta$-function given in \Cref{eq:principal Lfunction with zeta function}, we obtain, by taking absolute values inside the integral,
\begin{equation}
\begin{split}
\frac{1}{\vp(q)}|L_q^{(0)}(f)|^2
& \leq \frac{1}{\vp(q)}\rbr{\int_\R|f(x)|\abs{\zeta\rbr{1/2+ix}}\abs{\prod_{p|q}(1-p^{-\half-ix})}\dx}^2
\\
&\leq\frac{\norm{f\zeta(1/2+i\cdot)}_{L^1(\R)}^2}{\vp(q)} \prod_{p\mid q}(1+p^{-\half})^2
\\
&\leq \frac{2^{2\omega(q)}}{\vp(q)}\norm{f\zeta(1/2+i\cdot)}_{L^1(\R)}^2,
\end{split}
\end{equation}
where $\omega(q)=\sum_{p\mid q}1$ is the number of distinct prime factors of $q$.
The $L^1$ norm is finite because $f$ has compact support and $\zeta$ is analytic in a neighborhood of this support. In particular, $\zeta$ is bounded in the support of $f$. Then we have $\vp(q)=q\prod_{p|q}(1-p^{-1})$, where the product is over distinct prime factors of $q$ \cite[Theorem 62]{HaWr60a}. It follows that $\vp(q)\geq q 2^{-\omega(q)}$. Moreover, along the subsequence where $q$ is the product of the first $r$ primes and $r\to \infty$, we have
\begin{equation}
\omega(q)\lesssim \frac{\log(q)}{\log\log(q)},
\end{equation}
where $\lesssim$ means $\leq$ up to a multiplicative constant that is independent of $q$, see \cite[page 355]{HaWr60a}. By an easy estimate using the monotonicity of the RHS, we also obtain the same estimate for general $q$. Putting all this together, we obtain
\begin{equation}
\begin{split}
\frac{1}{\vp(q)}|L_q^{(0)}(f)|^2&\leq \norm{f\zeta(1/2+i\cdot)}_{L^1(\R)}^2\frac{2^{3\omega(q)}}{q}
\lesssim
\frac{2^{\frac{C\log(q)}{\log\log(q)}}}{q}
=q^{C\log(2)/\log\log(q)-1}
\overset{q\to \infty}{\longrightarrow}0,
\end{split}
\end{equation}
for some universal constant $C>0$ independent of $q$. Hence,  $\limsup_{q\to\infty}E_{q,M}^{(2)}=0$.

Next, we analyze the first term $E_{q,M}^{(1)}$. By  \Cref{lem:Fourier representation of Lf} and the fact that the expectation involves only a finite sum, we have
\begin{equation}
\begin{split}
E_{q,M}^{(1)}&:=\E\abr{\1_{\chi_q\neq \chi_0}\abs{\int_\R f(x)\sum_{n=M+1}^\infty \frac{\chi_q(n)}{n^{\half+ix}}\dx}^2}
\\
&\,\,=\lim_{L\to \infty}\E\abr{\1_{\chi_q\neq \chi_0}\abs{\sum_{n=M+1}^{Lq}\frac{\chi_q(n)}{n^{\half}}\hat{f}(\log(n))}^2}
\\
&\,\,=\lim_{L\to \infty}\sum_{n=M+1}^{Lq}\sum_{l=M+1}^{Lq}\frac{\E[\1_{\{\chi_q\neq\chi_0\}}\chi_q(n)\overline{\chi_q(l)}]}{n^{\half}l^{\half}}\hat{f}(\log(n))\overline{\hat{f}(\log(l))},
\end{split}
\end{equation}
where, without loss of generality, we have multiplied the upper limit of the sums by $q$ for the convenience of the analysis below. This does not change the result. Indeed, since the series converges, its partial sums converge along the sequence $\{i\}_{i\in \N}$, and thus also along the subsequence $\{iq\}_{i\in\N}$ for every $q$. It is convenient to write the series as a limit already at this stage, since it is not absolutely convergent. Otherwise, we would need to specify the interpretation of the resulting double series before doing any manipulations on it. 

Next, we use $\1_{\chi_q\neq\chi_0}=1-\1_{\chi_q=\chi_0}$ and apply the orthogonality relation \Cref{eq:orthogonality relation for Dirichlet characters} of the characters to evaluate the expectation. This yields
\begin{equation}
E_{q,M}^{(1)}=\lim_{L\to \infty}\sum_{n=M+1}^{Lq}\sum_{l=M+1}^{Lq}\frac{\1_{(n,q)=1}\1_{(l,q)=1}}{n^{\half}l^{\half}}\rbr{\1_{n\equiv l \mod q}-\frac{1}{\vp(q)}}\hat{f}(\log(n))\overline{\hat{f}(\log(l))}.
\end{equation}

We may also write $n=mq+r$ and $l=sq+t$ with $0<r,t<q$  (since Dirichlet characters vanish if their argument is divisible by $q$) and $m,s=0,1,\dots,L-1$. Furthermore, without loss of generality, we may assume that $q>M+1$ since we are taking $q\to \infty $ first. Thus, we obtain
\begin{equation}
\begin{split}
\label{eq:decomposition of the Lfunction variance}
E_{q,M}^{(1)}
&=\sum_{r,t=M+1}^{q-1}\rbr{\delta_{r,t}-\frac{1}{\vp(q)}}\frac{\1_{(r,q)=1}\1_{(t,q)=1}}{\sqrt{r}\sqrt{t}}\hat{f}(\log(r))\overline{\hat{f}(\log(t))}
\\
&\quad+\lim_{L\to\infty}\bigg\{\sum_{m=1}^{L-1}\sum_{r=1}^{q-1}\sum_{t=M+1}^{q-1}\rbr{\delta_{r,t}-\frac{1}{\vp(q)}}\frac{\1_{(r,q)=1}\1_{(t,q)=1}}{\sqrt{mq+r}\sqrt{t}}\hat{f}(\log(mq+r))\overline{\hat{f}(\log(t))}
\\
&\quad+\sum_{s=1}^{L-1}\sum_{t=1}^{q-1}\sum_{r=M+1}^{q-1}\rbr{\delta_{r,t}-\frac{1}{\vp(q)}}\frac{\1_{(r,q)=1}\1_{(t,q)=1}}{\sqrt{r}\sqrt{sq+t}}\hat{f}(\log(r))\overline{\hat{f}(\log(sq+t))}
\\
&\quad+\sum_{m=1}^{L-1}\sum_{s=1}^{L-1}\sum_{r,t=1}^{q-1}\rbr{\delta_{r,t}-\frac{1}{\vp(q)}}\frac{\1_{(r,q)=1}\1_{(t,q)=1}}{\sqrt{mq+r}\sqrt{sq+t}} \hat{f}(\log(mq+r))\overline{\hat{f}(\log(sq+t))}\bigg\}
\\
&=:S_M+\lim_{L\to\infty}(S_L(M)+\overline{S_L(M)}+S_{L,L}),
\end{split}
\end{equation}
with natural identifications of the sums on the last line. Note that we have again used the fact that $\1_{(xq+y,q)=1}=\1_{(y,q)=1}$ for all $x,y\in\N$.  

$S_M$ is the only term above that also requires the $M\to\infty$ limit to vanish.
It turns out that the terms $S_L(M),\overline{S_L(M)}$ and $S_{L,L}$ vanish already as we take the limit $q\to \infty$. The strategy is to show that each term above tends to zero in absolute value. Furthermore, for the last three terms $S_L(M)$, $\overline{S_L(M)}$ and $S_{L,L}$, we first show that the $L\to \infty$ limit of the absolute values is finite uniformly in $q$, allowing us to take the $q\to\infty$ limit termwise. Then we show that the logarithmic Fourier decay established in \Cref{eq:log bounds on the Fourier transform of $f$} suffices to ensure that each term vanishes as $q\to \infty$.

For $S_M$ we have
\begin{equation}
\label{eq:bound for SM}
\begin{split}
|S_M|\leq\sum_{r=M+1}^{q-1}\frac{\abs{\hat{f}(\log(r))}^2}{r}+\rbr{\frac{1}{\sqrt{\vp(q)}}\sum_{r=M+1}^{q-1}\frac{1}{\sqrt{r}}\abs{\hat{f}(\log(r))}}^2
\end{split}
\end{equation} 
As $q\to \infty$, the first term converges to the tail of a convergent series by the discussion after \Cref{lem:Fourier representation of Lf}, which vanishes as $M\to\infty$.

Without loss of generality, we may further assume that $M+1<\lfloor \sqrt{q-1}\rfloor$, so that we have 
\begin{equation}
\begin{split}
\frac{1}{\sqrt{\vp(q)}}\sum_{r=M+1}^{q-1}\frac{\abs{\hat{f}(\log(r))}}{\sqrt{r}}&=\frac{1}{\sqrt{\vp(q)}}\sum_{r=M+1}^{\lfloor\sqrt{q-1}\rfloor}\frac{\abs{\hat{f}(\log(r))}}{\sqrt{r}}
\\
&\quad+\frac{1}{\sqrt{\vp(q)}}\sum_{r=\lfloor\sqrt{q-1}\rfloor+1}^{q-1}\frac{\abs{\hat{f}(\log(r))}}{\sqrt{r}}
\end{split}
\end{equation}
We now consider each term above separately:
\begin{equation}
\begin{split}
\frac{1}{\sqrt{\vp(q)}}\sum_{r=M+1}^{\lfloor\sqrt{q-1}\rfloor}\frac{\abs{\hat{f}(\log(r))}}{\sqrt{r}}&\leq\frac{C}{\sqrt{\vp(q)}}\int_M^{\lfloor\sqrt{q-1}\rfloor}\frac{\dx}{\sqrt{x}}
\\
&\leq\frac{C\sqrt{\lfloor\sqrt{q-1}\rfloor}}{\sqrt{\vp(q)}}
\\
&\overset{*}{\lesssim}q^{-\half+\frac{\delta}{2}}\sqrt{\lfloor\sqrt{q-1}\rfloor}
\\
&\overset{q\to \infty}{\longrightarrow}0,
\end{split}
\end{equation}
for any $0<\delta<1/2$ and some constant $C>0$ independent of $q$. Above, we have used the constant bound for the Fourier transform from \Cref{eq:constant bound on the Fourier transform of $f$} and the decreasing property of $x\mapsto 1/\sqrt{x}$. The inequality marked with $*$ holds since $\vp(q)/q^{1-\delta}\to\infty$ for any $\delta>0$ \cite[Theorem 327]{HaWr60a}. Note that the constant $C$ may change from line to line, as we avoid introducing multiple constants that are not needed later. We continue with this convention when needed.

For the remaining term, we require a different lower bound for $\vp$, namely
\begin{equation}
\label{eq:log lower bound for vp}
\vp(q)>\frac{q}{e^{\gamma}\log\log(q)+\frac{3}{\log\log(q)}}\gtrsim \frac{q}{\log\log(q)}
\end{equation}
for $q\geq 3$ \cite[Theorem 8.8.7]{BaSh96a}. Here, $\gamma$ denotes the Euler-Mascheroni constant. Using \Cref{eq:log lower bound for vp}, \Cref{eq:log bounds on the Fourier transform of $f$}, and similar reasoning as above, particularly the decreasing property of $x\mapsto 1/(\log(x))$, we obtain  
\begin{equation}
\begin{split}
\frac{1}{\sqrt{\vp(q)}}\sum_{r=\lfloor\sqrt{q-1}\rfloor+1}^{q-1}\frac{\abs{\hat{f}(\log(r))}}{\sqrt{r}}&\leq \frac{C}{\sqrt{\vp(q)}}\sum_{r=\lfloor\sqrt{q-1}\rfloor+1}^{q-1}\frac{1}{\sqrt{r}\log(r)}
\\
&\leq \frac{C}{\sqrt{\vp(q)}}\frac{1}{\log(\lfloor\sqrt{q-1}\rfloor)}\int_{\lfloor\sqrt{q-1}\rfloor}^{q-1}\frac{\dx}{\sqrt{x}}
\\
&\leq \frac{C}{\sqrt{q}}\frac{\sqrt{\log\log(q)}}{\log(\lfloor\sqrt{q-1}\rfloor)}\sqrt{q-1}
\\
&\overset{q\to \infty}{\longrightarrow}0
\end{split}
\end{equation}
for some constant $C>0$, since $1/\log(\lfloor \sqrt{q-1}\rfloor)\lesssim (\half\log(q-1))^{-1}$. This concludes the proof that the limit of the $S_M$ term vanishes as $M\to \infty$.  

Next, let us estimate $S_{L,L}$. By writing out the Fourier transforms and interchanging the finite $r$ and $t$ sums with the integrals, we obtain
\begin{equation}
\begin{split}
S_{L,L}&:=\sum_{m=1}^{L-1}\sum_{s=1}^{L-1}\sum_{r,t=1}^{q-1}\rbr{\delta_{r,t}-\frac{1}{\vp(q)}}\frac{\1_{(r,q)=1}}{\sqrt{mq+r}}\frac{\1_{(t,q)=1}}{\sqrt{sq+t}}\hat{f}(\log(mq+r))\overline{\hat{f}(\log(sq+t))}
\\
&\,=\sum_{m=1}^{L-1}\sum_{s=1}^{L-1}\int_\R\int_\R\d x\d x'f(x)\overline{f(x')}\underbrace{\sum_{r,t=1}^{q-1}\1_{(r,q)=1}\1_{(t,q)=1}\rbr{\delta_{r,t}-\frac{1}{\vp(q)}}\frac{(mq+r)^{-ix}(sq+t)^{ix'}}{\sqrt{mq+r}\sqrt{sq+t}}}_{:=F_{m,s,q}(x,x')}.
\end{split}
\end{equation}
We now focus on the function $F$ and apply the first-order estimate 
\begin{equation}
\label{eq:first order Taylor estimate}
(1+a)^{-\half\pm iy}=1+\bigO(a)
\end{equation}
uniformly for $a>0$. The implied constant is uniform in $y$ for $y$ in a compact set $K\subset\R$. More precisely, the above notation means that, for all $a>0$ and $y\in K$,
\begin{equation}
\abs{(1+a)^{-\half\pm iy}-1}\leq C(K)a
\end{equation}
for some constant $C(K)$ depending only on the compact set $K$.

Recalling that $x,x'\in\supp(f)$, which is compact, we obtain
\begin{equation}
\begin{split}
F_{m,s,q}(x,x')=\frac{(mq)^{-ix}(sq)^{ix'}}{\sqrt{mq}\sqrt{sq}}\sum_{r,t=1}^{q-1}&\1_{(r,q)=1}\1_{(t,q)=1}\rbr{\delta_{r,t}-\frac{1}{\vp(q)}}(1+\bigO(r(mq)^{-1}))(1+\bigO(t(sq)^{-1}))
\\
=\frac{(mq)^{-ix}(sq)^{ix'}}{\sqrt{mq}\sqrt{sq}}\sum_{r,t=1}^{q-1}&\1_{(r,q)=1}\1_{(t,q)=1}\rbr{\delta_{r,t}-\frac{1}{\vp(q)}}\bigO(r(mq)^{-1})\bigO(t(sq)^{-1})
\end{split}
\end{equation}
since we can ignore the rest of the terms when we expand the $(1+\bigO(\cdot))$ terms. Indeed, these vanish upon summation over $r$ and $t$ by the first statement of \Cref{lem:estimates for r and t sums 1}. Furthermore, the second statement of the same lemma with $\sigma=1/2$ implies uniform boundedness in $q$ of the remaining sums multiplied by $q^{-1}$ extracted from the square roots. Since the $(\cdot)^{\pm iy}$, $y=x,x'$, terms are uniformly bounded, we conclude that
\begin{equation}
\sup_{x,x'\in\supp(f)}|F_{m,s,q}(x,x')|\leq C m^{-3/2}s^{-3/2},
\end{equation}
for some universal constant $C>0$ independent of $q$. This gives a uniform (in $q$) upper bound for each term in the sum $S_{L,L}:=\sum_{m,s=1}^LA_{m,s}(q)$ (here we introduce notation for the individual terms in the sum $S_{L,L}$) given by 
\begin{equation}
|A_{m,s}(q)|\leq C\frac{\norm{f}_{L^1(\R)}^2}{m^{\frac{3}{2}}s^{\frac{3}{2}}}:=B_{m,s},
\end{equation}
with the same constant $C>0$ as above. Furthermore, $\sum_{m,s=1}^\infty B_{m,s}<\infty$. Thus, the convergence is uniform in $q$, which allows us to take the $q\to \infty$ limit term by term in $S_{\infty,\infty}$. 

Our strategy is then to show that the logarithmic decay from the Fourier terms makes each term converge to zero. We aim to show that $\limsup_{q\to \infty}|A_{m,s}(q)|=0$ for all $m,s\in\N$. The terms $A_{m,s}(q)$ are defined by
\begin{equation}
A_{m,s}(q):=\sum_{r,t=1}^{q-1}\rbr{\delta_{r,t}-\frac{1}{\vp(q)}}\frac{\1_{(r,q)=1}}{\sqrt{mq+r}}\frac{\1_{(t,q)=1}}{\sqrt{sq+t}}\hat{f}(\log(mq+r))\overline{\hat{f}(\log(sq+t))}.
\end{equation}
Using the triangle inequality we have
\begin{equation}
\begin{split}
|A_{m,s}(q)|&\leq \sum_{r=1}^{q-1}\frac{1}{\sqrt{mq+r}\sqrt{sq+r}}|\hat{f}(\log(mq+r))||\hat{f}(\log(sq+r))|
\\
&\quad+\frac{1}{\sqrt{\vp(q)}}\sum_{r=1}^{q-1}\frac{|\hat{f}(\log(mq+r))|}{\sqrt{mq+r}}\frac{1}{\sqrt{\vp(q)}}\sum_{t=1}^{q-1}\frac{|\hat{f}(\log(sq+t))|}{\sqrt{sq+t}}
\\
&=:A_1(q)+A_2^{(1)}(q)A_2^{(2)}(q)
\end{split}
\end{equation}
with natural identifications of the shorthand notations. Then, using \Cref{eq:log bounds on the Fourier transform of $f$} again, we obtain for the first term
\begin{equation}
A_1(q)\leq C\sum_{r=1}^{q-1}\frac{1}{\sqrt{mq+r}\log(mq+r)}\frac{1}{\sqrt{sq+r}\log(sq+r)}
\end{equation}
for some universal constant $C>0$ independent of $q$.
Note that for $g(x)=1/\sqrt{x}$ or $g(x)=1/\log(x)$ we have $g(xq+y)\leq g(q)$ for all positive integers $x,y$. Thus, we obtain
\begin{equation}
A_1(q)\leq (q-1)\bigO([q\log(q)^2]^{-1})\overset{q\to \infty}{\longrightarrow}0,
\end{equation}
for fixed $m,s\geq 1$. Then it suffices to show that both $A_2^{(i)}(q)$, $i=1,2$, vanish individually in the limit $q\to \infty$. Note also that $g(xq+y)\leq g(y)$ for positive integers $x,q$ and $g$ as above. Therefore, we have 
\begin{equation}
|A_2^{(1)}(q)|\leq \frac{1}{\sqrt{\vp(q)}}\sum_{r=1}^{q-1}\frac{|\hat{f}(\log(mq+r))|}{\sqrt{r}}.
\end{equation}
This corresponds to the second term in the bound for $S_M$ in \Cref{eq:bound for SM}, with $M=0$ and the substitution $r\to mq+r$ in the Fourier term. Note that there we could have set $M=0$ from the outset. The argument applies here as well with minimal modification to show that $|A_2^{(1)}(q)|\to 0$ as $q\to \infty$ ($A_2^{(2)}(q)$ is completely analogous). We again split the sum at $\lfloor \sqrt{q-1} \rfloor$. Then, for the lower part, we use the constant bound for the Fourier term and, for the upper part, the logarithmic bound. Next, we again use $g(xq+y)\leq g(y)$ for $g(x)=1/\log(x)$. After this, the argument is exactly the same as in the analysis of $S_M$ above. Thus, $\limsup_{q\to \infty}A_{m,s}(q)=0$, and consequently $\limsup_{q\to \infty}S_{\infty,\infty}=0$.

Consider then $S_L(M)$; the case $\overline{S_L(M)}$ is similar. In this case, we may write
\begin{equation}
S_{L}(M)=\sum_{m=1}^L\int_\R\int_\R f(x)\overline{f(x')}F_{m,0,q}(x,x')\dx\dx',
\end{equation}
where 
\begin{equation}
\begin{split}
\label{eq:cancellation for SL terms}
F_{m,0,q}(x,x')&=\sum_{r=1}^{q-1}\sum_{t=M+1}^{q-1}\1_{(r,q)=1}\1_{(t,q)=1}\rbr{\delta_{r,t}-\frac{1}{\vp(q)}}\frac{(mq+r)^{-ix}t^{ix'}}{\sqrt{mq+r}\sqrt{t}}
\\
&=\frac{(mq)^{-ix}}{\sqrt{mq}}\sum_{t=M+1}^{q-1}\1_{(t,q)=1}t^{ix'-\half}\bigg(1+\bigO(t(mq)^{-1})
-\frac{1}{\vp(q)}\sum_{r=1}^{q-1}\1_{(r,q)=1}(1+\bigO(r(mq)^{-1}))\bigg),
\\
&=\frac{(mq)^{-ix}}{\sqrt{mq}}\sum_{t=M+1}^{q-1}\1_{(t,q)=1}t^{ix'-\half}\bigg(\bigO(t(mq)^{-1})
-\frac{1}{\vp(q)}\sum_{r=1}^{q-1}\1_{(r,q)=1}\bigO(r(mq)^{-1})\bigg)
\end{split}
\end{equation}
where we have again used the first-order Taylor estimate \Cref{eq:first order Taylor estimate} and cancelled out the terms constant in $r$. Note that we have suppressed the $M$ dependence from the shorthand notation since it plays no role in the following argument. We will also suppress it from other shorthand notations below.

Taking absolute values and applying the triangle inequality repeatedly, we may extend the lower limit of the $t$-summation from $M+1$ to $1$ to obtain an upper bound. After this, we may extract a common factor $m^{-3/2}$. Furthermore, the remaining sums are of the form
\begin{equation}
\frac{1}{\sqrt{q}}\sum_{t=1}^{q-1}\1_{(t,q)=1}t^{-\half}\rbr{\frac{t}{q}}^a \quad \text{ or } \quad \frac{1}{\vp(q)\sqrt{q}}\sum_{r,t=1}^{q-1}\1_{(t,q)=1}\1_{(r,q)=1}t^{-\half}\rbr{\frac{r}{q}}^a.
\end{equation}
These sums are uniformly bounded in $q$ by the third statement of \Cref{lem:estimates for r and t sums 1}, with $\sigma=1/2$. Thus, we have  
\begin{equation}
\sup_{x,x'\in\supp(f)}|F_{m,0,q}(x,x')|\leq Cm^{-3/2},
\end{equation}
for some universal $C>0$ independent of $q$. Writing $S_L(M):=\sum_{m=1}^LA_m(q)$, we have a uniform (in $q$) upper bound for each term
\begin{equation}
|A_m(q)|\leq C\frac{\norm{f}_{L^1(\R)}^2}{m^{\frac{3}{2}}}:=B_m
\end{equation}
with the same constant as above. Since $\sum_{m=1}^\infty B_m<\infty$, the convergence is again uniform in $q$, allowing us to take the $q\to \infty$ limit term by term in $S_\infty(M)$.  

The terms $A_{m}(q)$ are defined as follows:
\begin{equation}
A_{m}(q):=\sum_{r=1}^{q-1}\sum_{t=M+1}^{q-1}\rbr{\delta_{r,t}-\frac{1}{\vp(q)}}\frac{\1_{(r,q)=1}}{\sqrt{mq+r}}\frac{\1_{(t,q)=1}}{\sqrt{t}}\hat{f}(\log(mq+r))\overline{\hat{f}(\log(t))}.
\end{equation}
Next, we obtain the bound
\begin{equation}
\begin{split}
|A_m(q)|&\leq \sum_{r=1}^{q-1}\frac{|\hat{f}(\log(mq+r))||\hat{f}(\log(r))|}{\sqrt{mq+r}\sqrt{r}}
\\
&\quad + \frac{1}{\sqrt{\vp(q)}}\sum_{r=1}^{q-1}\frac{|\hat{f}(\log(mq+r))|}{\sqrt{mq+r}}\frac{1}{\sqrt{\vp(q)}}\sum_{t=1}^{q-1}\frac{|\hat{f}(\log(t))|}{\sqrt{t}}
\\
&=:A_1(q)+A_2^{(1)}(q)A_2^{(2)}(q).
\end{split}
\end{equation}
Here, we have again extended the lower limit of the $t$ sum from $M+1$ to $1$ to obtain an upper bound since all terms are non-negative. We have also used the same notation as in the analysis of $A_{m,s}$. For the first term, we have
\begin{equation}
A_1(q)\leq C\sum_{r=1}^{q-1}\frac{1}{\sqrt{mq+r}\log(mq+r)}\frac{1}{\sqrt{r}}\leq \frac{C}{\log(q)}\frac{1}{\sqrt{q}}\sum_{r=1}^{q-1}\frac{1}{\sqrt{r}}\overset{q\to \infty}{\longrightarrow}0.
\end{equation}
Convergence of the last expression is justified by \Cref{eq:sum for 1/sqrt r}, and the rest follows from arguments similar to those in the analysis of $A_{m,s}$. The terms $A_2^{(i)}(q)$, $i=1,2$, can be estimated similarly to those in the analysis of $A_{m,s}(q)$. Therefore,  $\limsup_{q\to \infty}A_{m}(q)=0$ and consequently $\limsup_{q\to \infty}S_{\infty}(M)=\limsup_{q\to \infty}\overline{S_{\infty}(M)}=0$. This finishes the proof.
\end{subsection}

\begin{subsection}{Proof of {{\Cref{prop:second claim of the proof of the main theorem}}}} 
\label{sec:proof of the second claim}
Since, by \Cref{lem:convergence of characters for general integers}, the random variables $\chi_{q}(n)$ converge jointly in distribution to $\omega_n$ as $q\to\infty$, the continuous mapping theorem (see for example \cite[Theorem 4.27]{Ka02a}) implies that it suffices to show the continuity of the function $F\colon 
\C^M\to \C$ defined by 
\begin{equation}
F(\zz):=\int_\R f(x)\sum_{n=1}^M\frac{z_n}{n^{\half+ix}}\dx,
\end{equation}
where $\zz=(z_1,z_2,\dots,z_M)$. This is clear since the sum is finite.
\end{subsection}

\begin{subsection}{Proof of {{\Cref{prop:third claim of the proof of the main theorem}}}}
To start, let us define 
\begin{equation}
\label{eq:def of PM}
P(M):=\{n\in\N\mid \text{every prime divisor of $n$ is at most $M$}\},
\end{equation}
and observe that 
\begin{equation}
\prod_{p\leq M}(1-\omega_p p^{-\half-ix})^{-1}=\sum_{n\in P(M)}\frac{\omega_n}{n^{\half+ix}}.
\end{equation}

Furthermore,
\begin{equation}
\sum_{n=1}^{M}\abs{\frac{\omega_n}{n^{\half+ix}}}\leq \sum_{n=1}^{M}n^{-\half}<\infty
\end{equation}
for a fixed $M$. Similarly,
\begin{equation}
\sum_{n\in P(M)}\abs{\frac{\omega_n}{n^{\half+ix}}}\leq \sum_{n\in P(M)}\frac{1}{n^{\half}}=\prod_{p\leq M}(1-p^{-\half})^{-1}<\infty
\end{equation}
for fixed $M$. Therefore, the sums involved in the following computation are deterministically bounded and uniformly so with respect to the integration variables. Thus, we have 
\begin{equation}
\begin{split}
E(M):=&\E\abr{\abs{\sum_{n=1}^{M}\frac{\omega_n\hat{f}(\log(n))}{n^{\half}}-\sum_{n\in P(m)}\frac{\omega_n\hat{f}(\log(n))}{n^{\half}}}^2}
\\
&=\sum_{n,m=1}^{M}\frac{\E[\omega_n\overline{\omega_m}]\hat{f}(\log(n))\overline{\hat{f}(\log(m))}}{n^\half m^\half}
\\
&\quad-\sum_{n=1}^{M}\sum_{m\in P(M)}\frac{\E[\omega_n\overline{\omega_m}]\hat{f}(\log(n))\overline{\hat{f}(\log(m))}}{n^\half m^\half}
\\
&\quad-\sum_{m=1}^{M}\sum_{n\in P(M)}\frac{\E[\omega_n\overline{\omega_m}]\hat{f}(\log(n))\overline{\hat{f}(\log(m))}}{n^\half m^\half}
\\
&\quad+\sum_{n\in P(M)}\sum_{m\in P(M)} \frac{\E[\omega_n\overline{\omega_m}]\hat{f}(\log(n))\overline{\hat{f}(\log(m))}}{n^\half m^\half}.
\end{split}
\end{equation} 

Recall from \Cref{lem:convergence of characters for general integers} that $\E[\omega_n\overline{\omega_m}]=\delta_{n,m}$. Thus, the above becomes
\begin{equation}
E(M)=\sum_{n=1}^{M}\frac{\abs{\hat{f}(\log(n))}^2}{n}-2\sum_{n=1}^{M}\1_{n\in P(M)}\frac{\abs{\hat{f}(\log(n))}^2}{n}+\sum_{n=1}^\infty\1_{n\in P(M)}\frac{\abs{\hat{f}(\log(n))}^2}{n}
\end{equation}
The first term converges to a finite limit by \Cref{eq:convergence of a series of Fourier transforms of f}. The indicator in the second term is always $1$, so the sum in the second term converges to the same finite value as the first term. The last term converges to the same value by monotone convergence.
\end{subsection}   
\end{section}

\begin{section}{Convergence for $\Re(s)>1/2$ in the space of analytic functions}
\label{sec:Convergence in the space of analytic functions}
In this section, we prove that to the right of the critical line, that is,  in the half-plane $D:=\{s\in \C\mid\Re(s)>1/2\}$, the convergence of the random Dirichlet $L$-functions $L(s,\chi_q)$ to the randomized zeta function occurs in the space of analytic functions, provided that the event $\{\chi_q=\chi_0\}$ is excluded.  This result itself is neither surprising nor novel. In fact, Bagchi already proved it for prime moduli $q$ \cite{Ba81a}, and Kowalski \cite[Section 3.4]{Ko21a} states that an analogue of Bagchi's theorem should be expected to hold for a wider class of $L$-functions arising in different arithmetic aspects. However, our proof showcases the method of reorganizing the double sum in the $L^2$-computation, analogously to \Cref{eq:decomposition of the Lfunction variance}. Other aspects of the argument are relatively standard.

Let $H(D)$ denote the space of analytic functions on $D$, equipped with the topology of uniform convergence on compact subsets of $D$. This topology is a Fr\'{e}chet topology generated by countably many seminorms 
\begin{equation}
\label{eq:seminorms on compact exhaustion}
\norm{f}_n:=\sup_{s\in K_n}|f(s)|
\end{equation} 
on compact sets $K_n$ with $\cup_n K_n=D$ and $K_n\subset \mathrm{Int}(K_{n+1})$. The collection $\{K_n\}_{n\in\N}$ is called a compact exhaustion of $D$. We can define a Fr\'{e}chet metric
\begin{equation}
\label{eq:Frechet metric}
d(f,g):=\sum_{n=1}^\infty 2^{-n}\frac{\norm{f-g}_n}{1+\norm{f-g}_n}
\end{equation}
for $H(D)$. The series defining the metric converges uniformly for all $f,g\in H(D)$. This means that limits and expectations can always be interchanged with the series. Therefore, it suffices to prove the analogues of \Cref{prop:first claim of the proof of the main theorem,prop:second claim of the proof of the main theorem,prop:third claim of the proof of the main theorem} with respect to the sup-norm on a fixed compact set $K\subset D$. By definition, for convergence in the locally uniform topology, it always suffices to show uniform convergence on an arbitrary fixed compact set. Furthermore, by the uniform convergence of the series defining the metric, it is also sufficient to verify the extra condition in \Cref{thm:approximation} for a fixed compact set with the supremum norm.

We can now state the main result of this section, whose proof closely follows the strategy used in the proof of our main result, \Cref{thm:main theorem}.
\begin{theorem}
\label{thm:convergence of the L function in the space of analytic functions}
Let $\chi_q\sim\Unif(\Xcal_q)$, where $\Xcal_q$ denotes the set of all Dirichlet characters modulo $q$, and define $L_q(s):=L(s,\chi_q)$. Then
\begin{equation}
L_q\1_{\{\chi_q\neq\chi_0\}}\overset{d}{\to}\zeta_{\mathrm{rand}}
\end{equation}
in  $H(D)$ as $q\to \infty$ along the prime numbers or more generally along any sequence $\{q_n\}_{n\in\N}$ in the class $Q$ defined in \Cref{rem:q sequences}. That is, we take $q_n$ to be such that the smallest prime factor in its prime decomposition tends to infinity as $n\to \infty$.

On $D$, the randomized zeta function $\zeta_{\mathrm{rand}}$ is defined pointwise by the randomized Euler product
\begin{equation}
\zeta_{\mathrm{rand}}(s)=\prod_{p}(1-\omega_pp^{-s})^{-1},
\end{equation}
where $\omega_p$ are i.i.d. uniform random variables on the unit circle $\T$.
\end{theorem}
\begin{proof}
Firstly, $\zeta_{\mathrm{rand}}$  
is indeed a random element in $H(D)$, and 
\begin{equation}
\zeta_{N,\mathrm{rand}}(s):= \prod_{p\leq N}(1-\omega_pp^{-s})^{-1}
\end{equation}
converges almost surely to $\zeta_{\mathrm{rand}}(s)$ in $H(D)$. 
Furthermore, an analogous convergence result holds for the series
\begin{equation}
\sum_{n=1}^\infty\frac{\omega_n}{n^s}
\end{equation}
and 
\begin{equation}
\Lcal_{M,\omega}(s):=\sum_{n=1}^{M}\frac{\omega_n}{n^{s}},
\end{equation}
where the random variables $\omega_n$ are defined in \Cref{eq:def of omegan for nonprime n}. For the above claims, see \cite[page 184]{La96a}. Also 
\begin{equation}
\Lcal_{M,q}(s):=\sum_{n=1}^M\frac{\chi_q(n)}{n^s}
\end{equation}
is clearly a random element in $H(D)$ for fixed $M$.

We proceed to state the three propositions analogous to \Cref{prop:first claim of the proof of the main theorem,prop:second claim of the proof of the main theorem,prop:third claim of the proof of the main theorem}.
\begin{proposition}
\label{prop:uniform L^2 estimate on arbitrary compact}
Let $\chi_q\sim \mathrm{Unif}(\Xcal_q)$, where $\Xcal_q$ denotes the set of all Dirichlet characters modulo $q$. Then 
\begin{equation}
\lim_{M\to\infty}\limsup_{q\to \infty}\E\abr{\rbr{\sup_{s\in K}\abs{\1_{\{\chi_q\neq\chi_0\}}L_q(s)-\Lcal_{M,q}(s)}}^2}=0
\end{equation}
for arbitrary compact $K\subset D$. Above, the $q$-limit is taken in the manner described in \Cref{thm:convergence of the L function in the space of analytic functions}.
\end{proposition}

\begin{proposition}
\label{prop:convergence of LcalMq to LcalMomega on fixed compact}
For fixed $M>0$, as $q\to \infty$ (in the manner described in \Cref{thm:convergence of the L function in the space of analytic functions}), 
\begin{equation}
\Lcal_{M,q}\overset{d}{\to}\Lcal_{M,\omega}
\end{equation} 
in $H(D)$.
\end{proposition}

\begin{proposition}
\label{prop:equivalence of the approximations for fixed compact}
Let $K\subset D$ be an arbitrary compact set. Then 
\begin{equation}
\lim_{M\to\infty}\E\abr{\rbr{\sup_{s\in K}\abs{\Lcal_{M,\omega}(s)-\zeta_{M,\mathrm{rand}}(s)}}^2}=0.
\end{equation}
\end{proposition}

Given these and the discussion before the statement of \Cref{thm:convergence of the L function in the space of analytic functions}, the result again follows by applying \Cref{thm:approximation} twice, this time with $X=H(D)$. 
\end{proof}

\begin{subsection}{Proof of {{\Cref{prop:uniform L^2 estimate on arbitrary compact}}}}
First, observe that  
\begin{equation}
\begin{split}
E_{q,M}&:=\E\abr{\rbr{\sup_{s\in K}|\1_{\{\chi_q\neq\chi_0\}}L_q(s)-\Lcal_{M,q}(s)|}^2}
\\
&\,\,\leq\E\abr{\rbr{\sup_{s\in K}\1_{\{\chi_q\neq\chi_0\}}|L_q(s)-\Lcal_{M,q}(s)|+\sup_{s\in K}\1_{\{\chi_q=\chi_0\}}|\Lcal_{M,q}(s)|}^2}
\\
&\,\,\leq\E\abr{\1_{\{\chi_q\neq\chi_0\}}\sup_{s\in K}|L_q(s)-\Lcal_{M,q}(s)|^2}+\P(\chi_q=\chi_0)\sup_{s\in K}|\Lcal_{M,q}^{(0)}(s)|^2
\\
&\,\,=:E_{q,M}^{(1)}+E_{q,M}^{(2)}.
\end{split}
\end{equation}
For the second term, with fixed $M>0$, we have
\begin{equation}
E_{q,M}^{(2)}\leq\frac{1}{\vp(q)}\sup_{s\in K}\rbr{\sum_{n=1}^Mn^{-\Re(s)}}^2\overset{q\to \infty}{\longrightarrow}0.
\end{equation} 

For the first term $E_{q,M}^{(1)}$, we proceed as follows. We choose a compact set $K'\subset D$, say a rectangle, such that $K\subset \mathrm{Int}(K')$, and apply the Cauchy integral formula to write
\begin{equation}
(L_q(s)-\Lcal_{M,q}(s))^2=\frac{1}{2\pi i}\oint_{\partial K'}\frac{(L_q(w)-\Lcal_{M,q}(w))^2}{w-s}\d w
\end{equation}  
Taking absolute values and noting that $|w-s|\geq\delta>0$ for all $s\in K$ and $w\in \partial K'$ with $\delta=\dist(\partial K',K)$, we obtain
\begin{equation}
|L_q(s)-\Lcal_{M,q}(s)|^2\leq C\oint_{\partial K'}|L_q(w)-\Lcal_{M,q}(w)|^2\d |w|,
\end{equation}
for some constant $C>0$ independent of $s$. Taking suprema over $s\in K$, multiplying by $\1_{\chi_q\neq \chi_0}$ and then taking expectations, we obtain
\begin{equation}
E_{q,M}^{(1)}\leq C\oint_{\partial K'}\E\abr{\1_{\chi_q\neq \chi_0}|L_q(w)-\Lcal_{M,q}(w)|^2}\d|w|.
\end{equation}
Note that the expectation is a finite sum and may consequently be interchanged with the integral. Thus, it remains to show that the integrand tends to zero pointwise and either justify the use of dominated convergence or show that the convergence is uniform on $\partial K'$ and use the fact that $K'$ is a finite rectangle. 

We proceed to show the uniform convergence. We write $s=\sigma+it\in\partial K'$ and obtain
\begin{equation}
\begin{split}
E_{q,M}(s)&:=\E\abr{\1_{\chi_q\neq \chi_0}|L_q(s)-\Lcal_{M,q}(s)|^2}
\\
&\,\,=\lim_{L\to \infty}\E\abr{\1_{\chi_q\neq \chi_0}\abs{\sum_{k=M+1}^{Lq}\frac{\chi_q(k)}{k^{\sigma+it}}}^2}
\\
&\,\,=\lim_{L\to \infty}\E\abr{\1_{\chi_q\neq \chi_0}\sum_{k,l=M+1}^{Lq}\frac{\chi_q(k)\overline{\chi_q(l)}}{k^{\sigma+it}l^{\sigma-it}}},
\end{split}
\end{equation}
where we have again used a convenient rewriting of the infinite sum, justified by the fact that the expectation is a finite sum. We can mimic the proof of \Cref{prop:first claim of the proof of the main theorem}. By using the identity $\1_{\chi_q\neq \chi_0}=1-\1_{\chi_q=\chi_0}$ and the orthogonality of the characters to compute the expectation term by term, we may write, analogously to \Cref{eq:decomposition of the Lfunction variance},
\begin{equation}
\begin{split}
\label{eq:splitting of the double sum in bulk}
E_{q,M}(s)&\,\,=\sum_{r,u=M+1}^{q-1}\rbr{\delta_{r,u}-\frac{1}{\vp(q)}}\frac{\1_{(r,q)=1}\1_{(u,q)=1}}{r^{\sigma+it}u^{\sigma-it}}
\\
&\qquad+ \lim_{L\to \infty}\bigg( \sum_{n=1}^{L-1}\sum_{r=1}^{q-1}\sum_{u=M+1}^{q-1}\rbr{\delta_{r,u}-\frac{1}{\vp(q)}}\frac{\1_{(r,q)=1}\1_{(u,q)=1}}{(nq+r)^{\sigma+it}u^{\sigma-it}}
\\
&\qquad+\sum_{m=1}^{L-1}\sum_{u=1}^{q-1}\sum_{r=M+1}^{q-1}\rbr{\delta_{r,u}-\frac{1}{\vp(q)}}\frac{\1_{(r,q)=1}\1_{(u,q)=1}}{r^{\sigma+it}(mq+u)^{\sigma-it}}
\\
&\qquad+\sum_{n,m=1}^{L-1}\sum_{r,u=1}^{q-1}\rbr{\delta_{r,u}-\frac{1}{\vp(q)}}\frac{\1_{(r,q)=1}\1_{(u,q)=1}}{(nq+r)^{\sigma+it}(mq+u)^{\sigma-it}}\bigg)
\\
&=:S_M(s)+\lim_{L\to\infty}(S_L(M,s)+\overline{S_L(M,s)}+S_{L,L}(s)).
\end{split}
\end{equation}
We shall estimate each term separately and show that they tend to zero in absolute value. We have 
\begin{equation}
\abs{S_M(s)}\leq \sum_{r=M+1}^{q-1}r^{-2\sigma}+\rbr{\frac{1}{\sqrt{\vp(q)}}\sum_{r=M+1}^{q-1}r^{-\sigma}}^2
\end{equation}
Since $\sigma>1/2$, the first term converges to zero as we take $q\to \infty$ followed by $M\to \infty$.

Next, we estimate for $1/2<\sigma<1$
\begin{equation}
\frac{1}{\sqrt{\vp(q)}}\sum_{r=M+1}^{q-1}r^{-\sigma}\leq \frac{C}{\sqrt{\vp(q)}}\int_M^{q-1}x^{-\sigma}\dx\overset{*}{\leq} \frac{C}{\sqrt{\vp(q)}}q^{1-\sigma}\lesssim q^{1/2-\sigma+\frac{\delta}{2}}\overset{q\to\infty}{\longrightarrow}0,
\end{equation}
for any $0<\delta <|1/2-\sigma|/2$ and $\sigma>1/2$. The last bound holds since $\vp(q)/q^{1-\delta}\to\infty$ for any $\delta>0$ \cite[Theorem 327]{HaWr60a}. For $\sigma=1$ the step $(*)$ above produces $C\log(q)/\sqrt{\vp(q)}$ and for $\sigma>1$ simply $C/\sqrt{\vp(q)}$. Both of these also converge to $0$.

In this analysis, we may replace $\sigma$ with its minimum value $\sigma_{\min}:=\min_{\partial K'\ni s=\sigma+it}(\sigma)$, which is attained on the left vertical edge of $K'$, ensuring uniform convergence $|S_M(s)|\to 0$.

Note that a first-order Taylor estimate similar to \Cref{eq:first order Taylor estimate} also holds with the replacement $1/2+ix\to s\in D$, and the implied constant is uniform in $s$ on compact subsets of $D$. Consider then the term $|S_{L,L}(s)|$ given by
\begin{equation}
\begin{split}
|S_{L,L}(s)|&=\bigg|\sum_{n,m=1}^{L-1}(nq)^{- \sigma-it}(mq)^{- \sigma+it}
\\
&\qquad\times\sum_{r,u=1}^{q-1}\1_{(r,q)=1}\1_{(u,q)=1}\rbr{\delta_{r,u}-\frac{1}{\vp(q)}}(1+\bigO(r(nq)^{-1}))
(1+\bigO(u(mq)^{-1}))\bigg|.
\end{split}
\end{equation}
By the first claim in \Cref{lem:estimates for r and t sums 1}, we may again disregard all terms except $\bigO(\cdot)\bigO(\cdot)$ when expanding the product of the $(1+\bigO(\cdot))$ factors. Furthermore, applying the triangle inequality and the second claim of \Cref{lem:estimates for r and t sums 1} yields the upper bound 
\begin{equation}
\begin{split}
\label{eq:estimate for SLL out of critical line}
\sup_{s\in \partial K'}\abs{S_{\infty,\infty}(s)}&\lesssim q^{-2(\sigma_{\min}-\half)}
\sum_{n,m=1}^\infty n^{-\sigma_{\min}-1}m^{-\sigma_{\min}-1}<\infty.
\end{split}
\end{equation} 

Finally, we turn to the analysis of the term $|S_L(M)|$. By rewriting the sums analogously to \Cref{eq:cancellation for SL terms} to obtain the necessary cancellation, using similar reasoning as before, and applying the third claim of \Cref{lem:estimates for r and t sums 1}, we obtain the bound
\begin{equation}
\label{eq:estimate for SL out of critical line}
\sup_{s\in \partial K'}|S_\infty(M,s)|\lesssim q^{-(\sigma_{\min}-\half)}\sum_{n=1}^\infty n^{-\sigma_{\min}-1}<\infty.
\end{equation}

Taking $q\to \infty$ in \Cref{eq:estimate for SLL out of critical line,eq:estimate for SL out of critical line} and noting that $\sigma_{\min}>1/2$, we conclude that
\begin{equation}
\lim_{M\to \infty}\limsup_{q\to \infty} E_{q,M}^{(1)}=0.
\end{equation}
This completes the proof.
\end{subsection}

\begin{subsection}{Proof of {{\Cref{prop:convergence of LcalMq to LcalMomega on fixed compact}}}}
The strategy is exactly the same as in the proof of \Cref{prop:second claim of the proof of the main theorem}. This time, we consider the map $\T^M\to H(D)$, $z\mapsto \sum_{n=1}^Mz_n n^{-s}$, and show that it is continuous. This is straightforward to verify. Indeed, let $K\subset D$ be an arbitrary compact set and let $z,w\in\T^M$. Then we have
\begin{equation}
\sup_{s\in K}\abs{\sum_{n=1}^M\frac{z_n}{n^s}-\sum_{n=1}^M\frac{w_n}{n^s}}\leq M\norm{z-w}\sup_{s\in K}\sum_{n=1}^M\frac{1}{n^{\Re(s)}}\overset{z\to w}{\longrightarrow}0.
\end{equation}
Since $K$ was arbitrary, this establishes the result via the continuous mapping theorem \cite[Theorem 4.27]{Ka02a}.
\end{subsection}

\begin{subsection}{Proof of {{\Cref{prop:equivalence of the approximations for fixed compact}}}}
First, fix $s\in D$. As in the proof of \Cref{prop:third claim of the proof of the main theorem}, both $\Lcal_{M,\omega}$ and $\zeta_{M,\mathrm{rand}}$ are deterministically bounded. Thus, recalling the notation defined in \Cref{eq:def of PM}, we have
\begin{equation}
\begin{split}
E_{M}(s)&:=\E\abr{|\Lcal_{M,\omega}(s)-\zeta_{M,\mathrm{rand}}(s)|^2}
\\
&=\sum_{k,l=1}^\infty\E[\omega_k\overline{\omega_l}]\bigg(\1_{\{k,l\leq M\}}-\1_{\{k\leq M\}}\1_{\{l\in P(M)\}}-\1_{\{k\in P(M)\}}\1_{\{l\leq M\}}
\\
&\qquad\qquad +\1_{\{k\in P(M)\}}\1_{\{l\in P(M)\}}\bigg)k^{-s}l^{-\overline{s}}
\\
&=\sum_{k=1}^{M}\frac{1}{k^{2\sigma}}-2\sum_{k=1}^{M}\1_{k\in P(M)}\frac{1}{k^{2\sigma}}+\sum_{k=1}^\infty\1_{k\in P(M)}\frac{1}{k^{2\sigma}}
\\
&\overset{M\to \infty}{\longrightarrow}0
\end{split}
\end{equation}
since $\sigma:=\Re(s)>1/2$.

Next, choosing $K'$ as in the proof of \Cref{prop:uniform L^2 estimate on arbitrary compact} and proceeding similarly, we obtain
\begin{equation} 
\E\abr{\rbr{\sup_{s\in K}|\Lcal_{M,\omega}(s)-\zeta_{M,\mathrm{rand}}(s)|}^2}\leq C\oint_{\partial K'}E_{M}(w)\d |w|.
\end{equation}
The claim follows by dominated convergence, using the bound 
\begin{equation}
\sup_{w\in \partial K'}|E_{M}(w)|\leq 4\sum_{k=1}^\infty \frac{1}{k^{1+\delta(K')}}<\infty,
\end{equation}
where $\delta(K')=\dist(K',\{\Re(s)=1/2\})=\sigma_{\min}-1/2$ with the earlier notation.
\end{subsection}

\end{section}

\section*{Acknowledgements}
The author is grateful to Eero Saksman and Christian Webb for insightful comments on earlier versions of this article.


\begin{thebibliography}{10}

\bibitem{Ap76a}
T.~M. Apostol.
\newblock {\em Introduction to analytic number theory}.
\newblock Undergraduate texts in mathematics. Springer-Verlag, New York, 1976.

\bibitem{ArBeHa17a}
L.-P. Arguin, D.~Belius, and A.~J. Harper.
\newblock Maxima of a randomized {R}iemann zeta function, and branching random
  walks.
\newblock {\em Ann. Appl. Probab.}, 27(1):178--215, 2017.
\newblock \href {https://doi.org/10.1214/16-AAP1201}
  {\path{doi:10.1214/16-AAP1201}}.

\bibitem{ArCr25a}
L.-P. Arguin and N.~Creighton.
\newblock Upper bounds on large deviations of {D}irichlet {$L$}-functions in
  the $q$-aspect.
\newblock {\em J. Number Theory}, 273:96--158, Aug. 2025.
\newblock URL: \url{http://dx.doi.org/10.1016/j.jnt.2025.01.009}, \href
  {https://doi.org/10.1016/j.jnt.2025.01.009}
  {\path{doi:10.1016/j.jnt.2025.01.009}}.

\bibitem{ArHaKi22a}
L.-P. Arguin, L.~Hartung, and N.~Kistler.
\newblock High points of a random model of the {R}iemann-zeta function and
  {G}aussian multiplicative chaos.
\newblock {\em Stochastic Process. Appl.}, 151:174--190, 2022.
\newblock \href {https://doi.org/10.1016/j.spa.2022.04.017}
  {\path{doi:10.1016/j.spa.2022.04.017}}.

\bibitem{AsJoKu11a}
K.~Astala, P.~Jones, A.~Kupiainen, and E.~Saksman.
\newblock Random conformal weldings.
\newblock {\em Acta Math.}, 207(2):203--254, 2011.
\newblock \href {https://doi.org/10.1007/s11511-012-0069-3}
  {\path{doi:10.1007/s11511-012-0069-3}}.

\bibitem{BaSh96a}
E.~Bach and J.~Shallit.
\newblock {\em Algorithmic number theory, volume 1: Efficient algorithms}.
\newblock MIT Press, 1996.

\bibitem{Ba81a}
B.~Bagchi.
\newblock {\em Statistical Behaviour and Universality Properties of the
  {R}iemann Zeta Function and Other Allied {D}irichlet Series}.
\newblock PhD thesis, Indian Statistical Institute, 1981.
\newblock URL:
  \url{https://www.proquest.com/dissertations-theses/statistical-behaviour-universality-properties/docview/2606903465/se-2}.

\bibitem{BaKe22a}
E.~C. Bailey and J.~P. Keating.
\newblock Maxima of log-correlated fields: some recent developments.
\newblock {\em J. Phys. A}, 55(5):053001, jan 2022.
\newblock URL: \url{https://dx.doi.org/10.1088/1751-8121/ac4394}, \href
  {https://doi.org/10.1088/1751-8121/ac4394}
  {\path{doi:10.1088/1751-8121/ac4394}}.

\bibitem{Be17a}
N.~Berestycki.
\newblock An elementary approach to {G}aussian multiplicative chaos.
\newblock {\em Electron. Commun. Probab.}, 22:Paper No. 27, 12, 2017.
\newblock \href {https://doi.org/10.1214/17-ECP58}
  {\path{doi:10.1214/17-ECP58}}.

\bibitem{BuEvLe25a}
H.~M. Bui, N.~Evans, S.~Lester, and K.~Pratt.
\newblock Weighted central limit theorems for central values of
  {$L$}-functions.
\newblock {\em J. Eur. Math. Soc. (JEMS)}, 27(6):2477--2529, 2025.
\newblock \href {https://doi.org/10.4171/jems/1417}
  {\path{doi:10.4171/jems/1417}}.

\bibitem{Co01a}
J.~B. Conrey.
\newblock {\em {$L$}-Functions and Random Matrices}, pages 331--352.
\newblock Springer Berlin Heidelberg, Berlin, Heidelberg, 2001.
\newblock \href {https://doi.org/10.1007/978-3-642-56478-9_14}
  {\path{doi:10.1007/978-3-642-56478-9_14}}.

\bibitem{DaKuRh16a}
F.~David, A.~Kupiainen, R.~Rhodes, and V.~Vargas.
\newblock {L}iouville quantum gravity on the {R}iemann sphere.
\newblock {\em Comm. Math. Phys.}, 342(3):869--907, 2016.
\newblock \href {https://doi.org/10.1007/s00220-016-2572-4}
  {\path{doi:10.1007/s00220-016-2572-4}}.

\bibitem{DuRhSh14b}
B.~Duplantier, R.~Rhodes, S.~Sheffield, and V.~Vargas.
\newblock Critical {G}aussian multiplicative chaos: convergence of the
  derivative martingale.
\newblock {\em Ann. Probab.}, 42(5):1769--1808, 2014.
\newblock \href {https://doi.org/10.1214/13-AOP890}
  {\path{doi:10.1214/13-AOP890}}.

\bibitem{DuSh11a}
B.~Duplantier and S.~Sheffield.
\newblock {L}iouville quantum gravity and {KPZ}.
\newblock {\em Invent. Math.}, 185(2):333--393, 2011.
\newblock \href {https://doi.org/10.1007/s00222-010-0308-1}
  {\path{doi:10.1007/s00222-010-0308-1}}.

\bibitem{Ea26a}
A.~Earnst.
\newblock Non-vanishing of {D}irichlet {$L$}-functions at the central point
  with restricted root number, 2026.
\newblock URL: \url{https://arxiv.org/abs/2603.22124}, \href
  {https://arxiv.org/abs/2603.22124} {\path{arXiv:2603.22124}}.

\bibitem{FyHiKe12a}
Y.~V. Fyodorov, G.~A. Hiary, and J.~P. Keating.
\newblock Freezing transition, characteristic polynomials of random matrices,
  and the {R}iemann zeta function.
\newblock {\em Phys. Rev. Lett.}, 108(17):170601, Apr 2012.
\newblock \href {https://doi.org/10.1103/PhysRevLett.108.170601}
  {\path{doi:10.1103/PhysRevLett.108.170601}}.

\bibitem{FyKe14a}
Y.~V. Fyodorov and J.~P. Keating.
\newblock Freezing transitions and extreme values: random matrix theory, and
  disordered landscapes.
\newblock {\em Philos. Trans. R. Soc. Lond. Ser. A Math. Phys. Eng. Sci.},
  372(2007):20120503, 32, 2014.
\newblock \href {https://doi.org/10.1098/rsta.2012.0503}
  {\path{doi:10.1098/rsta.2012.0503}}.

\bibitem{Ga24a}
P.~Gao.
\newblock Bounds for moments of {D}irichlet {$L$}-functions to a fixed modulus.
\newblock {\em Math. Z.}, 307(4):Paper No. 70, July 2024.
\newblock URL: \url{http://dx.doi.org/10.1007/s00209-024-03541-5}, \href
  {https://doi.org/10.1007/s00209-024-03541-5}
  {\path{doi:10.1007/s00209-024-03541-5}}.

\bibitem{GaZh25a}
P.~Gao and L.~Zhao.
\newblock Upper bounds for moments of {D}irichlet {$L$}-functions to a fixed
  modulus.
\newblock {\em Mathematika}, 71(4):Paper No. e70052, Sept. 2025.
\newblock URL: \url{http://dx.doi.org/10.1112/mtk.70052}, \href
  {https://doi.org/10.1112/mtk.70052} {\path{doi:10.1112/mtk.70052}}.

\bibitem{HaWr60a}
G.~H. Hardy and E.~M. Wright.
\newblock {\em An introduction to the theory of numbers}.
\newblock Clarendon Press, Oxford, fourth edition, 1960.

\bibitem{Ha13b}
A.~J. Harper.
\newblock Sharp conditional bounds for moments of the {R}iemann zeta function,
  2013.
\newblock URL: \url{https://arxiv.org/abs/1305.4618}, \href
  {https://arxiv.org/abs/1305.4618} {\path{arXiv:1305.4618}}.

\bibitem{HeRaSo19a}
W.~Heap, M.~Radziwi{\l}{\l}, and K.~Soundararajan.
\newblock Sharp upper bounds for fractional moments of the {R}iemann zeta
  function.
\newblock {\em Q. J. Math.}, 70(4):1387--1396, Sept. 2019.
\newblock URL: \url{http://dx.doi.org/10.1093/qmathj/haz027}, \href
  {https://doi.org/10.1093/qmathj/haz027} {\path{doi:10.1093/qmathj/haz027}}.

\bibitem{HsWo20a}
P.-H. Hsu and P.-J. Wong.
\newblock On {S}elberg's central limit theorem for {D}irichlet {$L$}-functions.
\newblock {\em J. Th{\'e}or. Nombres Bordeaux}, 32(3):685--710, 2020.
\newblock URL: \url{https://www.numdam.org/articles/10.5802/jtnb.1139/}, \href
  {https://doi.org/10.5802/jtnb.1139} {\path{doi:10.5802/jtnb.1139}}.

\bibitem{JuSaWe20a}
J.~Junnila, E.~Saksman, and C.~Webb.
\newblock Imaginary multiplicative chaos: moments, regularity and connections
  to the {I}sing model.
\newblock {\em Ann. Appl. Probab.}, 30(5):2099--2164, 2020.
\newblock \href {https://doi.org/10.1214/19-AAP1553}
  {\path{doi:10.1214/19-AAP1553}}.

\bibitem{Ka85a}
J.-P. Kahane.
\newblock Sur le chaos multiplicatif.
\newblock {\em Ann. Sci. Math. Qu\'ebec}, 9(2):105--150, 1985.

\bibitem{Ka02a}
O.~Kallenberg.
\newblock {\em Foundations of modern probability}.
\newblock Probability and its applications. Springer, New York, second edition,
  2002.

\bibitem{KeSn00b}
J.~Keating and N.~Snaith.
\newblock Random matrix theory and {$L$}-functions at $s= 1/2$.
\newblock {\em Comm. Math. Phys.}, 214:91--100, 2000.

\bibitem{KeSn00a}
J.~P. Keating and N.~C. Snaith.
\newblock Random matrix theory and {$\zeta(1/2+it)$}.
\newblock {\em Comm. Math. Phys.}, 214(1):57--89, 2000.
\newblock \href {https://doi.org/10.1007/s002200000261}
  {\path{doi:10.1007/s002200000261}}.

\bibitem{KeSn03a}
J.~P. Keating and N.~C. Snaith.
\newblock Random matrices and {$L$}-functions.
\newblock {\em J. Phys. A}, 36(12):2859--2881, mar 2003.
\newblock URL: \url{https://dx.doi.org/10.1088/0305-4470/36/12/301}, \href
  {https://doi.org/10.1088/0305-4470/36/12/301}
  {\path{doi:10.1088/0305-4470/36/12/301}}.

\bibitem{Ko21a}
E.~Kowalski.
\newblock {\em An introduction to probabilistic number theory}, volume 192 of
  {\em Cambridge Studies in Advanced Mathematics}.
\newblock Cambridge University Press, Cambridge, 2021.
\newblock \href {https://doi.org/10.1017/9781108888226}
  {\path{doi:10.1017/9781108888226}}.

\bibitem{LaRhVa15a}
H.~Lacoin, R.~Rhodes, and V.~Vargas.
\newblock Complex {G}aussian multiplicative chaos.
\newblock {\em Comm. Math. Phys.}, 337(2):569--632, 2015.
\newblock \href {https://doi.org/10.1007/s00220-015-2362-4}
  {\path{doi:10.1007/s00220-015-2362-4}}.

\bibitem{La96a}
A.~Laurin{\v c}ikas.
\newblock {\em Limit Theorems for the Riemann Zeta-Function}, volume 352 of
  {\em Mathematics and its Applications}.
\newblock Kluwer Academic Publishers, Dordrecht, 1996.

\bibitem{Ma15a}
T.~Madaule.
\newblock Maximum of a log-correlated {G}aussian field.
\newblock {\em Ann. Inst. Henri Poincar\'e Probab. Stat.}, 51(4):1369--1431,
  2015.
\newblock URL: \url{https://www.numdam.org/articles/10.1214/14-AIHP633/}, \href
  {https://doi.org/10.1214/14-AIHP633} {\path{doi:10.1214/14-AIHP633}}.

\bibitem{Ma74b}
B.~Mandelbrot.
\newblock Multiplications al{\'e}atoires it{\'e}r{\'e}es et distributions
  invariantes par moyenne pond{\'e}r{\'e}e al{\'e}atoire.
\newblock {\em C. R. Acad. Sci. Paris S\'er. A}, 278:289--292, 1974.

\bibitem{Ma74a}
B.~B. Mandelbrot.
\newblock Intermittent turbulence in self-similar cascades: divergence of high
  moments and dimension of the carrier.
\newblock {\em J. Fluid Mech.}, 62(2):331--358, 1974.
\newblock \href {https://doi.org/10.1017/S0022112074000711}
  {\path{doi:10.1017/S0022112074000711}}.

\bibitem{NaPaSi23a}
J.~Najnudel, E.~Paquette, and N.~Simm.
\newblock Secular coefficients and the holomorphic multiplicative chaos.
\newblock {\em Ann. Probab.}, 51(4):1193--1248, 2023.
\newblock \href {https://doi.org/10.1214/22-aop1616}
  {\path{doi:10.1214/22-aop1616}}.

\bibitem{Po20a}
E.~Powell.
\newblock Critical {G}aussian multiplicative chaos: a review.
\newblock {\em Markov Process. Related Fields}, 27(4):557--606, 2021.

\bibitem{RaSo12a}
M.~Radziwi{\l}{\l} and K.~Soundararajan.
\newblock Continuous lower bounds for moments of zeta and {$L$}-functions.
\newblock {\em Mathematika}, 59(1):119--128, Nov. 2012.
\newblock URL: \url{http://dx.doi.org/10.1112/S0025579312001088}, \href
  {https://doi.org/10.1112/s0025579312001088}
  {\path{doi:10.1112/s0025579312001088}}.

\bibitem{RaSo17a}
M.~Radziwi{\l}{\l} and K.~Soundararajan.
\newblock Selberg's central limit theorem for {$\log{|\zeta(1/2+it)|}$}.
\newblock {\em Enseign. Math.}, 63(1-2):1--19, 2017.
\newblock \href {https://doi.org/10.4171/LEM/63-1/2-1}
  {\path{doi:10.4171/LEM/63-1/2-1}}.

\bibitem{RhVa14a}
R.~Rhodes and V.~Vargas.
\newblock {G}aussian multiplicative chaos and applications: a review.
\newblock {\em Probab. Surv.}, 11:315--392, 2014.
\newblock \href {https://doi.org/10.1214/13-PS218}
  {\path{doi:10.1214/13-PS218}}.

\bibitem{RoVa10a}
R.~Robert and V.~Vargas.
\newblock {G}aussian multiplicative chaos revisited.
\newblock {\em Ann. Probab.}, 38(2):605--631, 2010.
\newblock \href {https://doi.org/10.1214/09-AOP490}
  {\path{doi:10.1214/09-AOP490}}.

\bibitem{RuSo05a}
Z.~Rudnick and K.~Soundararajan.
\newblock Lower bounds for moments of {$L$}-functions.
\newblock {\em Proc. Natl. Acad. Sci. USA}, 102(19):6837--6838, May 2005.
\newblock URL: \url{http://dx.doi.org/10.1073/pnas.0501723102}, \href
  {https://doi.org/10.1073/pnas.0501723102}
  {\path{doi:10.1073/pnas.0501723102}}.

\bibitem{RuSo06a}
Z.~Rudnick and K.~Soundararajan.
\newblock Lower bounds for moments of {$L$}-functions: symplectic and
  orthogonal examples.
\newblock In {\em Multiple {D}irichlet series, automorphic forms, and analytic
  number theory}, volume~75 of {\em Proc. Sympos. Pure Math.}, pages 293--303.
  Amer. Math. Soc., Providence, RI, 2006.
\newblock URL: \url{http://dx.doi.org/10.1090/pspum/075/2279944}, \href
  {https://doi.org/10.1090/pspum/075/2279944}
  {\path{doi:10.1090/pspum/075/2279944}}.

\bibitem{SaWe20a}
E.~Saksman and C.~Webb.
\newblock The {R}iemann zeta function and {G}aussian multiplicative chaos:
  statistics on the critical line.
\newblock {\em Ann. Probab.}, 48(6):2680--2754, 2020.
\newblock \href {https://doi.org/10.1214/20-AOP1433}
  {\path{doi:10.1214/20-AOP1433}}.

\bibitem{Se46a}
A.~Selberg.
\newblock Contributions to the theory of the {R}iemann zeta-function.
\newblock {\em Arch. Math. Naturvid.}, 48(5):89--155, 1946.

\bibitem{Se91a}
A.~Selberg.
\newblock {\em Collected papers}, volume~2, chapter 42. Old and new conjectures
  and results about a class of Dirichlet series.
\newblock Springer, 1991.

\bibitem{Sh16a}
A.~Shamov.
\newblock On {G}aussian multiplicative chaos.
\newblock {\em J. Funct. Anal.}, 270(9):3224--3261, 2016.
\newblock \href {https://doi.org/10.1016/j.jfa.2016.03.001}
  {\path{doi:10.1016/j.jfa.2016.03.001}}.

\bibitem{So09a}
K.~Soundararajan.
\newblock Moments of the {R}iemann zeta function.
\newblock {\em Ann. of Math. (2)}, 170(2):981--993, Sept. 2009.
\newblock URL: \url{http://dx.doi.org/10.4007/annals.2009.170.981}, \href
  {https://doi.org/10.4007/annals.2009.170.981}
  {\path{doi:10.4007/annals.2009.170.981}}.

\bibitem{Wu22a}
X.~Wu.
\newblock The fourth moment of {D}irichlet {$L$}-functions at the central
  value.
\newblock {\em Math. Ann.}, 387(3-4):1199--1248, Sept. 2022.
\newblock URL: \url{http://dx.doi.org/10.1007/s00208-022-02483-9}, \href
  {https://doi.org/10.1007/s00208-022-02483-9}
  {\path{doi:10.1007/s00208-022-02483-9}}.

\bibitem{Yo11a}
M.~Young.
\newblock The fourth moment of {D}irichlet {$L$}-functions.
\newblock {\em Ann. of Math. (2)}, 173(1):1--50, Jan. 2011.
\newblock URL: \url{http://dx.doi.org/10.4007/annals.2011.173.1.1}, \href
  {https://doi.org/10.4007/annals.2011.173.1.1}
  {\path{doi:10.4007/annals.2011.173.1.1}}.

\bibitem{Za19a}
R.~Zacharias.
\newblock Mollification of the fourth moment of {D}irichlet {$L$}-functions.
\newblock {\em Acta Arith.}, 191(3):201--257, 2019.
\newblock URL: \url{http://dx.doi.org/10.4064/aa170901-16-10}, \href
  {https://doi.org/10.4064/aa170901-16-10} {\path{doi:10.4064/aa170901-16-10}}.

\end{thebibliography}
\end{document}